\documentclass{amsart}
\usepackage[utf8]{inputenc}
\usepackage[english]{babel}
\usepackage{hyperref}
\usepackage{ae,aecompl}
\usepackage[margin=1in]{geometry}
\usepackage{graphicx}
\usepackage{amsmath, amsfonts, amsthm}
\usepackage{tikz}
\usetikzlibrary{decorations.markings}
\usepackage{ytableau}
\usepackage{caption}
\usepackage{comment}
\usepackage{microtype}
\usepackage{mathrsfs}
\usepackage{enumitem}
\graphicspath{ {graphics/} }
\usepackage{minibox}

\theoremstyle{definition}
\newtheorem{thm}{Theorem}[section]
\newtheorem{lem}[thm]{Lemma}
\newtheorem{df}[thm]{Definition}

\newtheorem{prop}[thm]{Proposition}
\newtheorem{cor}[thm]{Corollary}

\newtheorem{rem}[thm]{Remark}

\newcommand{\HH}{\mathbb{H}}
\newcommand{\C}{\mathbb{C}}
\newcommand{\R}{\mathbb{R}}
\newcommand{\Z}{\mathbb{Z}}
\newcommand{\N}{\mathbb{N}}
\newcommand{\E}{\mathbb{E}}
\newcommand{\Var}{\text{Var}}
\newcommand{\ii}{\mathbf{i}}
\newcommand{\bfy}{\mathbf{y}}
\newcommand{\bfz}{\mathbf{z}}
\newcommand{\mm}{\mathfrak{m}}
\newcommand{\B}{\mathfrak{B}}
\newcommand{\V}{\mathcal{V}}
\newcommand{\SSS}{\mathscr{S}}
\newcommand{\F}{\mathbb{F}}
\newcommand{\PP}{\mathcal{P}}
\newcommand{\W}{\mathcal{W}}
\newcommand{\EE}{\mathbb{E}}
\newcommand{\trig}{\textrm{trig}}

\newcommand{\corners}{\textrm{corners}}
\newcommand{\sym}{\textrm{sym}}

\newcommand{\interior}[1]{
  {\kern0pt#1}^{\mathrm{o}}
}

\numberwithin{equation}{section}

\begin{document}

\title{Universal behavior of the corners of orbital beta processes}

\author{Cesar Cuenca}

\address{Department of Mathematics, MIT, Cambridge, MA, USA}
\email{cuenca@mit.edu}

\date{}

\begin{abstract}
There is a unique unitarily-invariant ensemble of $N\times N$ Hermitian matrices with a fixed set of real eigenvalues $a_1 > \dots > a_N$.
The joint eigenvalue distribution of the $(N-1)$ top-left principal submatrices of a random matrix from this ensemble is called the orbital unitary process.
There are analogous matrix ensembles of symmetric and quaternionic Hermitian matrices that lead to the orbital orthogonal and symplectic processes, respectively.
By extrapolation, on the dimension of the base field, of the explicit density formulas, we define the orbital beta processes.
We prove the universal behavior of the virtual eigenvalues of the smallest $m$ principal submatrices, when $m$ is independent of $N$ and the eigenvalues $a_1 > \dots > a_N$ grow linearly in $N$ and in such a way that the rescaled empirical measures converge weakly.
The limiting object is the Gaussian beta corners process.
As a byproduct of our approach, we prove a theorem on the asymptotics of multivariate Bessel functions.
\end{abstract}

\maketitle

\section{Introduction}

\subsection{Preface}

A \textit{beta ensemble} (of rank $N$) is a probability measure on $N$-tuples $(x_1 \geq x_2 \geq \ldots \geq x_N)$ of non-increasing real numbers with density, with respect to the Lebesgue measure,
\begin{equation}\label{def:betaensemble}
\frac{1}{Z} \prod_{1 \leq i < j \leq N}{|x_i - x_j|^{\beta}}\prod_{k=1}^{N}{w(x_k)},
\end{equation}
where $\beta > 0$ is a parameter, $w : \R \rightarrow [0, \infty)$ is a weight function, and $Z$ is a normalization constant, typically called the partition function.
The weight function has to decay sufficiently fast at infinity for the partition function to be finite.
An example is $w(x) = \exp(-x^2/2)$: the Gaussian weight.
When $w$ is the Gaussian weight and $\beta = 1, 2, 4$, the beta ensemble $(\ref{def:betaensemble})$ is the eigenvalue distribution for the Gaussian Orthogonal/Unitary/Symplectic ensemble (GOE/GUE/GSE) from random matrix theory, respectively; see the books \cite{AGZ}, \cite{F}, \cite{Me}.
In a different direction, beta ensembles are related to the mean field theory of physical systems with log-gas interactions, and $\beta$ is the inverse temperature of the model.
Furthermore, beta ensembles with a general $\beta > 0$ are intimately related to Selberg-type integrals, e.g. see the survey \cite{FW}.
These considerations show the importance of considering the most general case $\beta > 0$ of beta ensembles.

Our interest is in \textit{robust} (or \textit{universal}) limits of beta ensembles as $N$ tends to infinity, where by robust we mean that the limit measures depend very little on the specific features of the ensembles in question (for example, limiting behavior of $(\ref{def:betaensemble})$ should not depend much on the choice of weight function $w$). 
In some special cases, for instance $\beta = 2$ and $w(x) = \exp(-x^2/2)$, very explicit formulas for the correlation measures of $(\ref{def:betaensemble})$ allow limit transitions as $N \rightarrow \infty$ that yield the Sine process and Airy process in the bulk and the edge, respectively.
Of course, the general case $\beta > 0$ (and general $w$) is more complicated and finer techniques are required to obtain limits, and to prove they are robust.
Nevertheless, in recent history there has been spectacular progress in understanding the local bulk and edge limits of general beta ensembles, see for example \cite{DE}, \cite{ES}, \cite{RRV}, \cite{VV}, etc.
The universality of the resulting limits has also been highly studied, e.g. \cite{Sh}, \cite{BFG}, \cite{BEY}, \cite{GH}.
Likewise, global Gaussian asymptotics (and the appearance of the Gaussian Free Field) is of recent interest and many articles are devoted to it: see \cite{J2}, \cite{BGu}, \cite{BG}, \cite{GZ}, \cite{CE}, \cite{ES}, etc.

In this paper we study a different kind of beta ensemble which is \textit{not} of the form $(\ref{def:betaensemble})$, but is still an extrapolation of eigenvalue densities of certain random matrix ensembles called the \textit{orbit measures}; e.g., see \cite{De} and references therein.
An orbit measure is the pushforward of the Haar measure of a (locally compact) Lie group $G$ to an adjoint orbit of its Lie algebra $\mathfrak{g}$, hence the name. We shall be interested in $G = U(N), O(N), Sp(N)$, as we can have the rank $N$ tend to infinity.
When $G = U(N)$, then $\mathfrak{g}$ is isomorphic to the Lie algebra of $N\times N$ Hermitian matrices, denoted $\textrm{Herm}(N)$, and any orbit measure is supported by Hermitian matrices $M = [M_{i, j}]_{i, j = 1}^N$ with a fixed set of eigenvalues.
By projecting the orbit measure on the set of eigenvalues of the top-left principal submatrices $M^{(k)} = [M_{i, j}]_{i, j = 1}^k$, $1 \leq k \leq N-1$, we obtain a random triangular array $\{x^{(k)}_i\}_{1 \leq i \leq k \leq N-1}$ whose distribution we call the orbital unitary process.
Similar considerations lead us to the orbital orthogonal/symplectic processes and, by analytic continuation, we arrive at the definition of \textit{orbital beta processes} in Definition $\ref{def:betaprocess}$.

Starting with the GUE/GOE/GSE, the same procedure of projecting to the eigenvalues of the top-left principal submatrices gives the Gaussian orthogonal/unitary/symplectic corners processes.
Again there is a one-parameter extrapolation called the \textit{Gaussian beta corners process} that has already been considered (see \cite{Ba}, \cite{JN}) and has found several nice applications, e.g. \cite{GoM}, \cite{GS}.
The Gaussian beta corners processes are convex combinations of orbital beta processes, so the latter are the more fundamental objects.
Thus one would imagine that the limits and universality results for Gaussian beta ensembles (and corners processes) are also present, in some form, for the orbital beta processes, but to the author's knowledge, there have been no such results in the literature.

We study the universality of the $m(m+1)/2$--dimensional \textit{corners process} of the orbital beta process; in the setting above, the corners process is the distribution of the particles $\{ x^{(k)}_i \}_{1 \leq i \leq k \leq m}$.
The regime of interest is when $m$ is fixed and independent of $N$; we show that the universal limit is the Gaussian beta corners process.
The universality in this scenario means that the limit is always the same, regardless of how the orbits tend to infinity.
Even though this is the first article where the universality in this limit regime is studied, in the special case $\beta = 2$, corner processes and their Gaussian limits are ubiquitous in both the discrete and continuous settings, for example in connection to: last-passage percolation \cite{Ba}, GUE minors \cite{JN}, lozenge tilings \cite{OR2}, \cite{GP}, and the 6-vertex model \cite{G}, \cite{Di}.

We proceed with a more detailed account on the setting and results of this paper.

\subsection{Orthogonal/Unitary/Symplectic orbital processes}

Let $\mathbb{F}$ be the field $\R, \C$ or the skew-field of quaternions $\HH$.
Let $M_{N\times N}(\F)$ be the set of $N\times N$ matrices with entries in $\F$.
Given $A\in M_{N\times N}(\F)$, denote
\begin{equation*}
A^* := \left \{
  \begin{aligned}
    &\textrm{transpose of }A && \text{if } \F = \R;\\
    &\textrm{Hemitian transpose of }A && \text{if } \F = \C;\\
    &\textrm{quaternionic conjugate transpose of }A && \textrm{if } \F = \HH.
\end{aligned} \right.
\end{equation*}
A matrix $A\in M_{N\times N}(\F)$ with $A = A^*$ is said to be self-adjoint; in all three cases, self-adjoint matrices have real spectra.
We want to consider ensembles of self-adjoint matrices $\{M_N\}_{N \geq 1}$ of growing size, such that each $M_N$ has fixed eigenvalues and random eigenvectors.

Consider the groups
\begin{equation*}
U(N, \F) := \left \{
  \begin{aligned}
    &O(N, \R) \ (\textrm{orthogonal group}) && \text{if } \F = \R;\\
    &U(N, \C) \ (\textrm{unitary group}) && \text{if } \F = \C;\\
    &Sp(N) \ (\textrm{compact symplectic group}) && \textrm{if } \F = \HH.
\end{aligned} \right.
\end{equation*}
Each one of them is a compact Lie group with a unique left- and right-translation invariant probability measure called the Haar distribution on $U(N, \F)$.

Define
$$\mathcal{P}_N := \{ a(N) = (a_1, \ldots, a_N)\in\R^N \mid a_1 > a_2 > \dots > a_N \}$$
as the set of ordered eigenvalues of a generic $N\times N$ self-adjoint matrix.
Given $a(N) = (a_1 > \ldots > a_N)\in\PP_N$, let $D(a(N))$ be the diagonal $N\times N$ matrix whose diagonal entries $D(a(N))_{i, i}$ are the entries $a_i$ of the $N$-tuple $a(N)$.
If $U_N \in U(N, \mathbb{F})$ is a random Haar-distributed matrix, then $M_N := U_ND(a(N))U_N^*$ is a random, self-adjoint matrix, with fixed real eigenvalues $a_1 > \dots > a_N$.

For $k = 1, 2, \ldots, N$, let $M_N^{(k)}$ be the top-left $k\times k$ corner of $M_N = U_ND(a(N))U_N^* = [m_{i, j}]_{i, j = 1}^N$, i.e., $M_N^{(k)} = [m_{i, j}]_{i, j = 1}^k$.
The matrices $M_N^{(k)}$ are self-adjoint and have random real eigenvalues $a^{(k)}_1 \geq \dots \geq a^{(k)}_k$.
The well-known Rayleigh's theorem (see \cite{B}, \cite{Ga}) asserts that the eigenvalues of different corners interlace:
$$a^{(k+1)}_1 \geq a^{(k)}_1 \geq a^{(k+1)}_2 \geq \dots \geq a^{(k+1)}_k \geq a^{(k)}_k \geq a^{(k+1)}_{k+1}, \textrm{ for all } 1 \leq k \leq N-1.$$
This theorem holds for deterministic matrices.
For our random ensemble of matrices $\{M_N = U_ND(a(N))U_N^*\}$ all inequalities are strict, almost surely.
We can arrange the eigenvalues of all corners $M_N^{(1)}, M_N^{(2)}, \dots, M_N^{(N)} = M_N$ in a triangular array:
\begin{align*}
&a_1 \qquad\quad a_2 \qquad\quad a_3 \qquad\qquad\qquad\ a_{N-1} \qquad\quad a_N\\
&\qquad a_1^{(N-1)} \quad  a_2^{(N-1)} \qquad\qquad a_{N-2}^{(N-1)} \quad\ a_{N-1}^{(N-1)}\\
&\qquad\qquad\quad\cdots\qquad\quad\cdots\qquad\quad\cdots\\
&\qquad\qquad\qquad\qquad a_1^{(2)} \qquad a_2^{(2)}\\
&\qquad\qquad\qquad\qquad\qquad a_1^{(1)}
\end{align*}

Following \cite{N}, we say that an arrangement as the one above is a \textit{Rayleigh triangle} if each number is smaller than its top-left neighbor, but larger than its top-right neighbor.
The $k$-tuple $a^{(k)} = (a^{(k)}_1, \ldots, a^{(k)}_k)$ is said to be the $k$-th level of the Rayleigh triangle, while $a = (a_1, \ldots, a_N)$ is the top level.
The probability distribution of the eigenvalues $\{a_j^{(k)}\}_{1 \leq j \leq k \leq N-1}$ in $\R^{\frac{N(N-1)}{2}}$ will be called \textit{orbital orthogonal/unitary/symplectic process} when $\F = \R/\C/\HH$, respectively.

A natural question is to study the (probabilistic) features of the Rayleigh triangles as $N$ goes to infinity and the top level $a(N) = (a_1 > a_2 > \dots > a_N)$ grows in a certain specific way.
In this article, we focus on the bottom $m$ levels of the random Rayleigh triangles, i.e., on the eigenvalues $\{a^{(k)}_j\}_{1 \leq j \leq k \leq m}$, when $m$ is independent of $N$.
We show that, under certain conditions on the growth of the top levels $a(N)$, a proper re-scaling of the random eigenvalues $a^{(m)}_1 > \dots > a^{(m)}_m$ converges weakly to the Gaussian beta ensemble with $\beta = \textrm{dim}(\F)$; see Theorem \ref{thm:GBE} for the precise statement.
A slight extension of this result shows that the re-scaled set of eigenvalues $a^{(k)}_j$, $1 \leq j \leq k \leq m$, of the bottom $m$ levels, converges to the Gaussian beta corners process with $\beta = \dim(\F)$: see Theorem $\ref{thm:GBE2}$ in the text.

Our results prove the \textit{universality} of the eigenvalues of orbital orthogonal/unitary/symplectic ensembles, and identifies the Gaussian beta  ensemble as the limiting object (the parameter $\beta$ takes values in $\{ 1, 2, 4 \}$ in this random matrix theory setting).
In fact, we go further and study a natural one-parameter generalization of the orbital processes above, that we call the \textit{orbital beta processes}.
The additional parameter is denoted $\theta$ and takes values in $(0, \infty)$; when $\theta = \frac{1}{2}, 1, 2$, we recover the orbital orthogonal, unitary and symplectic processes, respectively\footnote{To conform with traditional ``beta'' random matrix theory, it would make sense to use the parameter $\beta = 2\theta$ instead. We prefer the parameter $\theta$ that conforms with the theory of special functions (Jack polynomials and multivariate Bessel functions) which lie at the heart of our proofs.}.
The orbital beta processes are also probability measures on Rayleigh triangles $\{a_j^{(k)}\}_{1 \leq j \leq k \leq N-1}$ with a fixed top level $a = (a_1 > \dots > a_N)$.
One can think of the values $\{a^{(k)}_j\}_{1 \leq j \leq k}$ as ``virtual eigenvalues'' of the $k\times k$ corner, however the random matrix picture is lost for general values of $\theta > 0$.
We prove that the virtual eigenvalues of the bottom $m$ levels of the orbital beta processes still display a universal behavior and the limiting object is the general Gaussian beta corners process.

\subsection{Orbital beta processes}

Let $a = (a_1 > a_2 > \dots > a_N)\in\PP_N$ be an ordered $N$-tuple.
The \textit{interlacing scheme} $\PP(a) \subset \R^{N-1}\times\R^{N-2}\times\cdots\times \R$ is the convex polytope defined by
\begin{multline*}
\PP(a) := \{(a^{(N-1)}, \dots, a^{(1)}) \mid a^{(k)} \in \PP_k \textrm{ for } k = 1, 2, \dots, N-1, \\
a^{(k+1)}_j > a^{(k)}_j > a^{(k+1)}_{j+1} \textrm{ for all }1 \leq j \leq k \leq N-2, \textrm{ and } \\
a_j > a^{(N-1)}_j > a_{j+1} \textrm{ for all }1 \leq j \leq N-1 \}.
\end{multline*}
In other words, $\PP(a)$ is an embedding of the space of Rayleigh triangles with top level $a$ into $\R^{N-1}\times\R^{N-2}\times\cdots\times\R$.
The orbital orthogonal/unitary/symplectic processes with top level $a$ are probability distributions on $\PP(a)$.
They are absolutely continuous and their densities were calculated in \cite{N}\footnote{In \cite{N}, the author remarks that the proof is essentially within the calculations of \cite[II.9]{GN}.} for all three cases $\theta\in\{\frac{1}{2}, 1, 2\}$; see also \cite{Ba} for $\theta = 1$.
The formulae in \cite{N} can be extrapolated in the dimension of the base field, i.e., there is a formula depending continuously on a parameter $\theta$ that reduces to the densities for orbital orthogonal/unitary/symplectic processes when $\theta = \frac{1}{2}/1/2$, respectively.
This is the base of our definition of orbital beta processes.

\begin{df}\label{def:betaprocess}
Given $a = (a_1 > \dots > a_N)\in\PP_N$, the \textit{orbital beta process with top level $a$} is the probability measure $\mm^{a, \theta}$ on the interlacing scheme  $\PP(a)$ with density
\begin{equation}\label{def:densityOBP}
\mm^{a, \theta}\{ (a^{(N-1)}, \dots, a^{(1)}) \} \propto
\prod_{k = 1}^{N-1}\left\{\prod_{1 \leq i < j \leq k}{|a_i^{(k)} - a_j^{(k)}|^{2 - 2\theta}}
\prod_{s = 1}^{k+1}\prod_{r = 1}^k{|a_s^{(k+1)} - a_r^{(k)}|^{\theta - 1}} \prod_{j=1}^k{da^{(k)}_j} \right\},
\end{equation}
for any $(a^{(N-1)}, \dots, a^{(1)}) \in \PP(a)$.
In the formula above, we set $a^{(N)}_s := a_s$ for all $1 \leq s \leq N$.
Also, $\propto$ means that we are omitting a normalization constant that is chosen in such a way that the total weight of $(\ref{def:densityOBP})$ is $1$.
The value of the normalization constant admits a closed form expression, as shown in Lemma $\ref{lem:normalization}$.
\end{df}

One obvious but important property of the orbital beta process is the following: conditioned on the $k$-th level $a^{(k)}$, the joint distribution of $a^{(1)}, \dots, a^{(k-1)}$ (the levels $1, \ldots, k-1$) is independent from the joint distribution of $a^{(k+1)}, \dots, a^{(N-1)}$ (the levels $k+1, \ldots, N-1$).
Moreover, conditioned on $a^{(k)}$, the joint distribution of $a^{(1)}, \dots, a^{(k-1)}$ is $\mm^{a^{(k)}, \theta}$.

The density in $(\ref{def:densityOBP})$ is identically equal to one when $\theta = 1$.
This means that $\mm^{a, \theta = 1}$ is the uniform probability distribution on the interlacing scheme $\PP(a)$.
Moreover, conditioned on the $k$-th level $a^{(k)}$, the distribution of the lower levels $1, \ldots, k-1$ is uniform on $\PP(a^{(k)})$.
This is called the \textit{Gibbs property}.
Because of this property, the occurrence of the Gaussian unitary corners process in the limit of Theorem $\ref{thm:GBE}$ is not surprising: the Gibbs measures were classified in \cite{OV}, and \cite{OR2} gives heuristics arguing why the Gaussian unitary corners process is the only plausible candidate out of all Gibbs measures to be the limit in the regime of Theorem $\ref{thm:GBE}$.
Similarly, one could make a heuristical argument explaining the occurrence of the Gaussian orthogonal/symplectic corners processes in Theorem \ref{thm:GBE} when $\theta = \frac{1}{2}, 2$.
Indeed, the necessary classification of probability measures with the \textit{$\theta$-Gibbs property} (as in Definition \ref{def:Gibbs} below) was recently obtained in \cite{Ni}, when $\theta = \frac{1}{2}, 2$.
All these articles motivated us to ask whether similar results hold for general $\theta > 0$, and to make the statements precise.

\subsection{Main result}\label{sec:intromain}

Let us recall the definition of the Gaussian beta ensemble, see e.g. \cite{AGZ}, \cite{F}, \cite{Me}, which will play the role of a universal object in our main theorem.

\begin{df}\label{def:betagaussian}
The \emph{Gaussian} (or \emph{Hermite}) \emph{beta ensemble} (of rank $m$) is the absolutely continuous probability measure $\mm^{(\theta)}_m$ on the \emph{Weyl chamber}
$$
\W_m := \{(x_1, \ldots, x_m)\in\R^m : x_1 \geq \ldots \geq x_m \},
$$
whose density is
$$\mm^{(\theta)}_m(x_1, \ldots, x_m)  := \frac{1}{Z_{m, \theta}}\prod_{1\leq i < j\leq m}{|x_i - x_j|^{2\theta}}\exp\left( -\frac{\theta}{2}\sum_{i=1}^m{x_i^2} \right),$$
and $Z_{m, \theta}$ is an appropriate normalization constant.
The value of $Z_{m, \theta}$ admits a closed form expression by Selberg's integral.
\end{df}

\begin{rem}
The normalization $\exp(-\theta x^2/2)$ in Definition \ref{def:betagaussian} was chosen so that the formulas in $(\ref{eqn:meanvariance})$ are independent of $\theta$.
\end{rem}

\begin{df}\label{def:regularsignatures}
A sequence of ordered tuples $\{ a(N) \in \PP_N\}_{N \geq 1}$ is said to be \emph{regular} if
\begin{enumerate}[label=(\alph*)]
	\item $|a(N)_1|, |a(N)_N| = O(N)$, as $N\rightarrow\infty$;
	\item let $\mu_N$ be the empirical measure of the points $\{a(N)_i/N\}_{1 \leq i \leq N}$, i.e.
$$\mu_N := \frac{1}{N}\sum_{i=1}^N{\delta_{a(N)_i/N}}.$$
There exists a (compactly supported) probability measure $\mu$ on $\R$ such that
$$\mu_N \to \mu \ \text{ weakly, as }N \rightarrow \infty.$$
We say that $\mu$ is the \emph{limiting measure} of the ordered tuples $\{a(N)\}_{N \geq 1}$.
\end{enumerate}
\end{df}

The following is the main theorem of this article.

\begin{thm}\label{thm:GBE}
Let $m \in \N_+$.
Let $\{a(N) \in \PP_N\}_{N \geq 1}$ be a regular sequence of ordered tuples with limiting measure $\mu$.
Let $\{a^{(k, N)}_j\}_{1 \leq j \leq k \leq N-1}$ be a random element from $\PP(a(N))$, which is distributed by the orbital beta process with top level $a(N)$, as in $(\ref{def:densityOBP})$; in particular, the $m$-th level of this element is $\{a^{(m, N)}_i\}_{1 \leq i \leq m}$.
Then, as $N\rightarrow\infty$, the $m$-dimensional vectors
\begin{equation*}
\left( \frac{a_1^{(m, N)} - N \E[\mu_N]}{\sqrt{N \Var[\mu]}}, \ldots, \frac{a^{(m, N)}_m - N\E[\mu_N]}{\sqrt{N \Var[\mu]}}  \right),
\end{equation*}
converge weakly to the Gaussian beta ensemble of Definition $\ref{def:betagaussian}$.

Above, for any probability measure $\nu$ on $\R$, we have used the notations
\begin{equation}\label{eqn:meanvariance}
\E[\nu] := \int_{\R}{t \nu(d t)}, \quad \Var[\nu] := \int_{\R}{t^2 \nu(dt)} - \E[\nu]^2.
\end{equation}
\end{thm}

\bigskip

We can actually say even more and give the limit joint distribution of the triangular array $\{ a_j^{(k)} \}_{1 \leq j \leq k \leq m}$ consisting of the bottom $m$ levels.
It turns out that, with the same normalization as in Theorem $\ref{thm:GBE}$, the triangular array converges in distribution to the Gaussian beta corners process; see Theorem $\ref{thm:GBE2}$ below.

\begin{rem}
In an earlier version of this paper, condition (b) in Definition \ref{def:regularsignatures} was replaced by the condition

(b') there exists a function $f: [0, 1] \to \R$ such that
$$\sum_{i=1}^N{\left| \frac{a(N)_i}{N} - f\left( \frac{i}{N} \right) \right|} = o(\sqrt{N}),\quad N \to \infty.$$
The old hypotheses on $f$ for Theorem \ref{thm:GBE} were that $f$ had to be weakly decreasing, piecewise $C^1$ function and $f'$ had to have left and right limits at points of discontinuity.
Then the old conclusion was that the vectors
\begin{equation*}
\left( \frac{a_1^{(m, N)} - N \E[f]}{\sqrt{N \Var[f]}}, \ldots, \frac{a^{(m, N)}_m - N\E[f]}{\sqrt{N \Var[f]}}  \right),
\end{equation*}
\begin{equation*}
\text{where}\quad\E[f] := \int_{0}^{1}{f(t) dt}, \quad \Var[f] := \int_{0}^{1}{f(t)^2 dt} - \E[f]^2,
\end{equation*}
converge weakly to the Gaussian beta ensemble, as $N\to\infty$.
The proof of this version of the theorem is essentially the same, with the only differences being minor technical estimates --- see Step 3 in Section \ref{sec:m1}.
The current versions of Definition \ref{def:regularsignatures} and Theorem \ref{thm:GBE} are more in line with the language used in the random matrix theory literature.
Besides, condition (b) is less restrictive than (b').
I am grateful to one of the referees for suggesting to work with the current version of Definition \ref{def:regularsignatures}.
\end{rem}

\subsection{Methodology}

The first step towards the study of orbital beta processes is to relate them to the \textit{multivariate Bessel functions} (MBFs).
The MBFs are certain analytic functions on $2N$ variables
\begin{align*}
\C^N \times \C^N &\rightarrow \C\\
((a_1, \ldots, a_N), (x_1, \ldots, x_N)) &\mapsto \B_{(a_1, \ldots, a_N)}(x_1, \ldots, x_N; \theta)
\end{align*}
that further depend on an additional complex parameter $\theta$.
Each MBF $\B_{(a_1, \ldots, a_N)}(x_1, \ldots, x_N; \theta)$ is symmetric with respect to the variables $a_1, \ldots, a_N$ and, independently, with respect to the variables $x_1, \ldots, x_N$.
They are defined as the common eigenfunctions of a system of differential equations related to the rational Calogero-Moser integrable system; see \cite{D}, \cite{Op}, \cite{dJ}, \cite{GK}, and Section $\ref{sec:bessels}$ for more details.
The MBFs are part of a large family of special functions associated to root systems, whose initial purpose was to create a theory of Fourier analysis on symmetric spaces without any group structure behind it; e.g., consult the recent survey \cite{An}.

It turns out that the density of the orbital beta process, in Definition $\ref{def:betaprocess}$, is present in an integral representation of the multivariate Bessel functions, see Proposition $\ref{prop:combinatorial}$ in the text.
This connection between orbital beta processes and MBFs will be key in our approach.
The study of limits of statistical mechanical models and random matrices via special functions (Schur and Macdonald polynomials, Heckman-Opdam functions, etc.) has been very fruitful; for example, see \cite{OkR}, \cite{BC}, \cite{BG}, \cite{J}, \cite{GP}.

The relationship between orbital beta processes and Bessel functions allows to compute observables for the orbital beta process; the remaining task is to compare them to similar observables for the Gaussian beta ensemble.
The problem is reduced to the study of asymptotics for multivariate Bessel functions, in a specific limit regime.
The asymptotic study of special functions (e.g. classical univariate Bessel functions) is a subject with a long history; see, e.g., \cite{C}, \cite{BE}, \cite{O}.
However, the limit regime of our interest appears to be unexplored:
given any $m\in\N_+$, a sequence $\{a(N) = (a(N)_1, \ldots, a(N)_N) \in \R^N\}_{N}$ of ordered tuples and a probability measure $\mu$ on $\R$ satisfying the same conditions as in Theorem \ref{thm:GBE}, our main result on limits of Bessel functions is
$$
\lim_{N\rightarrow\infty}{\B_{a(N)}\left( \frac{y_1}{\sqrt{N}}, \dots, \frac{y_m}{\sqrt{N}}, 0^{N-m}; \theta \right)}
\exp\left( -\sqrt{N}\cdot \E[\mu_N] \sum_{i=1}^m{y_i} \right)
= \exp\left( \frac{\Var[\mu]}{2\theta}\sum_{i=1}^m{y_i^2} \right).
$$
We show that the limit is uniform for $(y_1, \ldots, y_m)$ belonging to compact subsets of $\C^m$, though for the proof of Theorem \ref{thm:GBE}, only pointwise convergence would suffice; see Theorem \ref{thm:asymptoticbessels} and Section \ref{sec:asymptotics} for the details.
The asymptotics of multivariate Bessel functions in this regime are studied with (degenerations of) recent formulas proved by the author in \cite{Cu18a}, \cite{Cu18b}; see also \cite{GP} for the $\theta = 1$ case.

\subsection{Related literature}

The matrix ensembles $U_NDU_N^*$, where $D$ is a fixed real diagonal matrix and $U_N$ is a Haar-distributed orthogonal/unitary matrix, were studied recently in \cite{MM}.
That paper shows that the marginals of these matrix ensembles coincide in the limit with the marginals of GOE/GUE.
Our main theorem, in the special cases $\theta \in \{\frac{1}{2}, 1 \}$, might follow from the results of [MM], although the limit regime there is not the same as ours.

In \cite{GuM}, the authors studied the asymptotics of spherical orthogonal/unitary integrals.
As mentioned in the third remark in Section $\ref{sec:besselsdef}$, these spherical integrals are special cases of the multivariate Bessel functions, whose asymptotics we study.
The limit regime considered in \cite{GuM} is very similar to the limit regime in Theorem $\ref{thm:asymptoticbessels}$; however, neither their result nor ours is a consequence of the other.
It is possible that the results from \cite{GuM} (and in fact, a one-parameter generalization) can be obtained from the formulas in Section $\ref{sec:formulas}$.

In a different direction, Gorin and Panova studied a statistical mechanical model which should be thought of as the discrete, $\theta = 1$ version of the orbital beta process (see \cite[Theorem 5.1]{GP}).
Roughly speaking, they studied a lozenge tilings model which yields a Rayleigh triangle with integral coordinates if we focus only on lozenges of a specific type.
They proved the universality of the positions for a specific type of lozenges near the turning point of the limit shape, and identified the Gaussian unitary corners process as the universal limit.
The same result was proved in \cite{No}, though the author there uses different methods.
A $q$-deformation of this result was also proved in \cite{MP}; the main tool there was the theory of determinantal point processes.

\subsection{Organization of the article}

In Section 2, we compute the normalization constant of the orbital beta process and extend Theorem $\ref{thm:GBE}$ to a multilevel version.
We introduce the multivariate Bessel functions in Section 3 and calculate some Bessel function observables in Section 4.
Then in Section 5, we summarize several new results on multivariate Bessel functions. These results include new formulas and an asymptotic statement for the Bessel functions when their number of arguments tends to infinity but almost all of them are specialized to zero.
The proofs of the results in this section are postponed until the technical Sections 7 and 8.
In Section 6, we prove the main Theorem $\ref{thm:GBE}$, by using the Dunkl transform theory (generalizing the Fourier transform).

\subsection{Notation}\label{sec:notation}

\begin{itemize}
	\item $\N_+ := \{1, 2, \dots\}$ is the set of positive integers. The variable $N$ always denotes a positive integer.
	\item We use a fixed parameter $\theta > 0$ throughout the paper. In the random matrix literature, the parameter $\beta > 0$ is more common; they are related by $\beta = 2\theta$.
	\item An $N$-tuple of real numbers $a = (a_1, \dots, a_N)$  such that $a_1 > \dots > a_N$ is called an \emph{$N$-ordered tuple}. The set of all $N$-ordered tuples is denoted $\PP_N$.
	\item Given $a\in\PP_N$, denote $|a| := a_1 + \dots + a_N$.
	\item Given $b\in\PP_{N+1}$ and $a\in\PP_N$, we write $a \prec b$, or $b \succ a$, if $b_1 > a_1 > b_2 > \dots > b_N >  a_N > b_{N+1}$.
	\item For $m\in\N_+$,
$$\W_m := \{(x_1, \ldots, x_m)\in\R^m : x_1 \geq \dots \geq x_m\}$$
is called the \emph{Weyl chamber}, whereas
$$\W_{m, \corners} := \{ (x^{(k)}_i)_{1 \leq i \leq k \leq m} \in \R^{\frac{m(m+1)}{2}} : x_i^{(k+1)} \geq x_i^{(k)} \geq x_{i+1}^{(k+1)} \textrm{ for all }1 \leq i \leq k \leq m-1 \}$$
is called the \emph{Gelfand-Tsetlin polytope}.
Let $1 \leq p \leq m$. For any element $(x^{(k)}_i)_{1 \leq i \leq k \leq m}\in \W_{m, \corners}$, we say that $x^{(p)} := (x^{(p)}_j)_{1 \leq j \leq p} \in \W_p$ is its $p$-th level.
	\item Denote $(0^k) = (0, \ldots, 0)$ (string of $k$ zeroes) for any $k \geq 1$; if $k = 0$, then $(0^k)$ is the empty string.
	\item Oftentimes $i$ is the index of a sum or product, so we use the bold $\ii = \sqrt{-1}$ for the imaginary unit.
\end{itemize}

\subsection*{Acknowledgments}
This text is part of a larger project which includes an article \cite{BCG} in preparation.
Some of the statements appear in both texts, for instance Theorems $\ref{besselthm2}$--$\ref{besselthm3}$.
I am grateful to Florent Benaych-Georges, Alexei Borodin and Yi Sun for useful conversations, and a special thanks to Vadim Gorin for suggesting the project and for several helpful discussions.
I am also grateful to the referees for their careful reading of the first version and for their suggestions, which helped to greatly improve the exposition and content of this article.
The author was partially supported by NSF Grant DMS-1664619.

\section{Orbital beta processes and limits to Gaussian beta corners processes}

\subsection{Normalization of the orbital beta processes}

The partition function for the orbital beta process $\mm^{a, \theta}$ of Definition $\ref{def:betaprocess}$ admits a closed product form.

\begin{lem}\label{lem:normalization}
Given a fixed $a = (a_1 > a_2 > \dots > a_N)\in\PP_N$,
\begin{gather*}
\int_{a = a^{(N)} \succ a^{(N-1)} \succ \dots \succ a^{(1)}}
\prod_{k = 1}^{N-1}\left\{\prod_{1 \leq i < j \leq k}{|a_i^{(k)} - a_j^{(k)}|^{2 - 2\theta}}
\prod_{s = 1}^{k+1}\prod_{r = 1}^k{|a_s^{(k+1)} - a_r^{(k)}|^{\theta - 1}}
\prod_{j = 1}^k{da_j^{(k)}}
\right\}\\
= \frac{\Gamma(\theta)^{\frac{N(N+1)}{2}}}{\prod_{k=1}^{N}{\Gamma(k \theta)}}
\prod_{1 \leq i < j \leq N}{|a_i - a_j|^{2\theta - 1}},
\end{gather*}
where $a^{(N)} := a$, and the integration is over the $N(N-1)/2$ coordinates $\{ a^{(k)}_j \}_{1 \leq j \leq k \leq N-1}$, satisfying (denote $a^{(k)} = (a^{(k)}_1, \ldots, a^{(k)}_k)$ for all $k$ and recall that $\prec$ was defined in Section $\ref{sec:notation}$): $a^{(1)} \prec \dots \prec a^{(N-1)} \prec a^{(N)} = a$.
\end{lem}
\begin{proof}
Let us denote by $I$ the $N(N-1)/2$-dimensional integral in the left hand side of the identity.
We evaluate $I$ in $(N-1)$ steps, integrating over $a^{(1)}, a^{(2)}, \ldots, a^{(N-1)}$ in that order.

First fix $a^{(N-1)}, \ldots, a^{(2)}$; the resulting integral over $a^{(1)}$ can be computed from (a very special case of) Anderson's integral, \cite{A}, \cite[(2.24)]{FW}:
$$\int_{a^{(2)} \succ a^{(1)}}{|a^{(2)}_1 - a^{(1)}_1|^{\theta - 1}|a^{(2)}_2 - a^{(1)}_1|^{\theta - 1}da^{(1)}_1} = \frac{\Gamma(\theta)^2}{\Gamma(2\theta)}\cdot |a^{(2)}_1 - a^{(2)}_2|^{2\theta - 1}.$$
We can replace this into $I$, so that we are left with an integral over $a^{(2)}, a^{(3)}, \ldots, a^{(N-1)}$.
To proceed, fix $a^{(N-1)}, \ldots, a^{(3)}$ and integrate over $a^{(2)}$ by using again Anderson's integral:
$$\int_{a^{(3)} \succ a^{(2)}}{(a^{(2)}_1 - a^{(2)}_2) \prod_{s=1}^3\prod_{r=1}^2{|a^{(3)}_s - a^{(2)}_r|^{\theta - 1}} da^{(2)}_1da^{(2)}_2} = \frac{\Gamma(\theta)^3}{\Gamma(3\theta)} \prod_{1 \leq i<j \leq 3}{|a^{(3)}_i - a^{(3)}_j|^{2\theta - 1}}.$$
Iterating this step $(N-1)$ times in total yields the final result.
We remark that the special case of Anderson's integral that we are using (in the notation of \cite[(2.24)]{FW}) is when all parameters $s_1, s_2, \ldots$ equal $\theta$.
\end{proof}

\subsection{Gaussian beta corners process}

The definition of the Gaussian beta corners process is a one-parameter $\theta$--interpolation of the Gaussian unitary corners process studied in \cite{Ba}, \cite{JN} (corresponding to $\theta = 1$) and the analogous orthogonal/symplectic processes studied in \cite{N} (corresponding to $\theta = \frac{1}{2}, 2$).
The author of \cite{N} claims that all his calculations were already contained in \cite{GN}.

\begin{df}\label{def:corners}
The \textit{Gaussian} (or \textit{Hermite}) \textit{beta corners process} (with $m$ levels) $\mm^{(\theta)}_{m, \corners}$ is the absolutely continuous probability measure on the Gelfand-Tsetlin polytope $\mathcal{W}_{m, \corners}$ with density
\begin{multline}\label{Gaussiancorners}
\mm^{(\theta)}_{m, \corners}\left((x_i^{(k)})_{1\leq i < k\leq m}\right) :=
\frac{1}{Z_{m, \theta}^{\corners}}\prod_{1\leq i<j\leq m}{(x_i^{(m)} - x_j^{(m)})}\prod_{i=1}^m{\exp\left( -\frac{\theta (x_i^{(m)})^2}{2} \right)}\\
\times\prod_{k=1}^{m-1} \left\{ \prod_{1\leq i<j\leq k}{\left|x_i^{(k)} - x_j^{(k)} \right|^{2-2\theta}}\prod_{a=1}^k\prod_{b=1}^{k+1}{\left| x_a^{(k)} - x_b^{(k+1)} \right|^{\theta - 1}} \right\},
\end{multline}
where $Z_{m, \theta}^{\corners}$ is an appropriate normalization chosen so that the total weight of $\mm^{(\theta)}_{m, \corners}$ is equal to $1$. A closed form expression for $Z_{m, \theta}^{\corners}$ can be calculated from Lemma $\ref{lem:normalization}$ and Selberg's integral formula.
\end{df}

Let $1 \leq p \leq m$.
The projection $(x_i^{(k)})_{1 \leq i \leq k \leq m} \mapsto (x^{(\ell)}_j)_{1 \leq j \leq \ell \leq p}$ maps $\mm^{(\theta)}_{m, \corners}$ to $\mm^{(\theta)}_{p, \corners}$.
This result follows from a special case of the integral identity \cite[Sec. 2.2, B)]{N}.
As a consequence of this fact and Lemma $\ref{lem:normalization}$, for any $1 \leq p \leq m$, the projection $(x_i^{(k)})_{1 \leq i \leq k \leq m} \mapsto (x^{(p)}_j)_{1 \leq j \leq p}$ maps $\mm^{(\theta)}_{m, \corners}$  to the Gaussian beta ensemble $\mm^{(\theta)}_p$ of rank $p$, as in Definition $\ref{def:betagaussian}$.

\begin{df}\label{def:Gibbs}
A probability measure $\mm$ on $\mathcal{W}_{m, \corners}$, which is absolutely continuous with density $\mm\left( (x_i^{(k)})_{1 \leq i \leq k \leq m} \right)$, is said to have the \emph{$\theta$--Gibbs property} if:
\begin{itemize}
	\item For any $1 < n < m$ and conditioned on the $n$-th level $x^{(n)} = (x^{(n)}_i)_{1 \leq i \leq n}$, the joint distribution of the bottom $(n-1)$ levels $x^{(1)}, \dots, x^{(n-1)}$ is independent of the joint distribution of the top $(m - n)$ levels $x^{(n+1)}, \dots, x^{(m)}$.
	\item For any $1 \leq n < m$, and conditioned on the $(n+1)$-st level $x^{(n+1)} = (x^{(n+1)}_i)_{1 \leq i \leq n+1}$, the joint distribution of the first $n$ levels $x^{(1)}, \dots, x^{(n)}$ is given by the density
\begin{multline}\label{thetaGibbs}
\mm\left( x^{(1)}, \dots, x^{(n)} \ \Big| \ x^{(n+1)} \right) =
\frac{\prod_{k=1}^{n+1}{\Gamma(k\theta)}}{\Gamma(\theta)^{(n+1)(n+2)/2}} \prod_{1 \leq i < j \leq n+1}{\left( x^{(n+1)}_i - x^{(n+1)}_j \right)^{1 - 2\theta}}\\
\times \prod_{k=1}^{n}\left\{ \prod_{1 \leq i < j \leq k}\left( x^{(k)}_i - x^{(k)}_j \right)^{2 - 2\theta} \prod_{a=1}^k \prod_{b=1}^{k+1} \left| x^{(k)}_a - x^{(k+1)}_b \right|^{\theta - 1} \right\}.
\end{multline}
\end{itemize}
\end{df}

\begin{lem}\label{lem:theta}
The Gaussian beta corners process $\mm^{(\theta)}_{m, \corners}$ has the $\theta$--Gibbs property.
Also, for any $a\in\PP_m$, the orbital beta process $\mm^{a, \theta}$ has the $\theta$--Gibbs property.
\end{lem}
\begin{proof}
Both $\mm^{(\theta)}_{m, \corners}$ and $\mm^{a, \theta}$ trivially satisfy the first condition in the definition of $\theta$--Gibbs property, because if $x^{(n)}$ is fixed then their densities \eqref{Gaussiancorners} and \eqref{def:densityOBP} can be written as products of the form $f(x^{(1)}, \dots, x^{(n-1)})\cdot g(x^{(n+1)}, \dots, x^{(m)})$.

As for the second property, we need to do the same calculations for both $\mm^{(\theta)}_{m, \corners}$ and $\mm^{a, \theta}$; for concreteness, let us concentrate on $\mm^{a, \theta}$.
The distribution of $x^{(1)}, \dots, x^{(n)}$, conditioned on $x^{(n+1)}, \dots, x^{(m)}$, is given by the ratio with numerator \eqref{def:densityOBP} and denominator being the integral of \eqref{def:densityOBP} with respect to the variables $x^{(1)}, \dots, x^{(n)}$ on the domain $x^{(1)} \prec \dots \prec x^{(n)} \prec x^{(n+1)}$ ($x^{(n+1)}$ is fixed).
The integral can be calculated by using Lemma \ref{lem:normalization}.
The result is
\begin{equation}\label{densitycond1}
\mm^{a, \theta}\left( x^{(1)}, \dots, x^{(n)} \Bigg| x^{(n+1)}, \dots, x^{(m)} \right) = \text{the right hand side of \eqref{thetaGibbs}}.
\end{equation}
Next, by the first condition of the $\theta$-Gibbs property and a general property of conditional probability,
\begin{equation}\label{densitycond2}
\mm^{a, \theta}\left( x^{(1)}, \dots, x^{(n)} \Bigg| x^{(n+1)}, \dots, x^{(m)} \right) = \mm^{a, \theta}\left( x^{(1)}, \dots, x^{(n)} \Bigg| x^{(n+1)} \right).
\end{equation}
Both \eqref{densitycond1} and \eqref{densitycond2} prove the second condition of the $\theta$--Gibbs property for $\mm^{a, \theta}$.
\end{proof}

\subsection{Multilevel extension of Theorem $\ref{thm:GBE}$}

We can strengthen Theorem $\ref{thm:GBE}$ by studying, simultaneously, the levels $1, 2, \ldots, m$ of a random element of $\PP(a(N))$, distributed according to the orbital beta process.

\begin{thm}\label{thm:GBE2}
Assume we are in the setting of Theorem $\ref{thm:GBE}$.
Let $a^{(m, N)}, \dots, a^{(1, N)}$ be the levels $m, \dots, 1$ in a random element from $\PP(a(N))$, distributed by the orbital beta process.
In particular, each $\{a^{(k, N)}_i\}_{1\leq i\leq k\leq m}$ belongs to the Gelfand-Tsetlin polytope $\W_{m, \corners}$.
Then, as $N\rightarrow\infty$, the $m(m+1)/2$--dimensional vectors
\begin{equation}\label{rescaled}
\left( \frac{a^{(k, N)}_i - N \E[\mu_N]}{\sqrt{N \Var[\mu]}} \right)_{1\leq i\leq k\leq m}
\end{equation}
converge weakly to the Gaussian beta corners process of Definition $\ref{def:corners}$. 
\end{thm}
\begin{proof}[Deduction of Theorem \ref{thm:GBE2} from Theorem \ref{thm:GBE}]
From Lemma \ref{lem:theta}, the measure $\mm^{a(N), \theta}$ has the $\theta$-Gibbs property.
By the very definition of $\theta$-Gibbs property, the distribution of the vector $(a^{(k, N)}_i)_{1 \leq i \leq k \leq m}$ also has the $\theta$-Gibbs property and so does the distribution of the rescaled vector \eqref{rescaled}.
Assuming that Theorem \ref{thm:GBE} holds, the distribution of the $m$-th level
$$\left( \frac{a^{(m, N)}_i - N \E[\mu_N]}{\sqrt{N \Var[\mu]}} \right)_{1\leq i\leq m}$$
converges weakly to the Gaussian beta ensemble $\mm^{(\theta)}_m$ of rank $m$, as $N \to \infty$.
As a result, the $m(m+1)/2$--dimensional vectors \eqref{rescaled} converge weakly, as $N \to \infty$, to a measure $\mm$ on $\W_{m, \corners}$ such that (1) it satisfies the $\theta$-Gibbs property, and (2) the projection $(x^{(k)}_i)_{1 \leq i \leq k \leq m} \mapsto (x^{(m)}_i)_{1 \leq i \leq m}$ maps $\mm$ to $\mm^{(\theta)}_m$.
It is not hard to see that there is a unique measure on $\W_{m, \corners}$ satisfying both (1) and (2).
In fact, Lemma \ref{lem:theta} shows that $\mm^{(\theta)}_{m, \corners}$ is the unique measure satisfying (1) and (2); thus $\mm = \mm^{(\theta)}_{m, \corners}$, concluding the proof.
\end{proof}

\begin{rem}
Definition \ref{def:Gibbs} of the $\theta$-Gibbs property is equivalent to the definition given in \cite[Def. 2.15]{GS}.
It is also a continuous analogue of the definition of the Jack-Gibbs property in \cite[Def. 2.10]{GS}. Likewise, the proof of our Theorem \ref{thm:GBE2} is the continuous version of the proof of \cite[Prop. 2.16]{GS}.
\end{rem}

\section{Multivariate Bessel functions}\label{sec:bessels}

\subsection{Dunkl operators}

Consider the space of smooth, complex-valued functions on $N$ variables $x_1, \ldots, x_N$.
The \textit{Dunkl operators} (of type A) are defined by
$$\xi_i := \frac{\partial}{\partial x_i} + \theta\sum_{j \neq i}{\frac{1}{x_j - x_i}(1 - s_{ij})}, \  i = 1, 2, \ldots, N,$$
where $s_{ij}$ permutes the variables $x_i$ and $x_j$.
 The key property of these operators is their pairwise commutativity.

\begin{thm}\label{commutativity}\cite{D}
For any $i \neq j$, we have $\xi_i\xi_j = \xi_j\xi_i$.
\end{thm}

Because of Theorem $\ref{commutativity}$, for any polynomial $P(x_1, \ldots, x_N)$ the operator $P(\xi_1, \ldots, \xi_N)$ is unambiguously defined.
Moreover if $P$ is a \textit{symmetric} polynomial on $N$ variables, and we restrict $P(\xi_1, \ldots, \xi_N)$ to the space of symmetric, smooth $N$-variate functions, then $P(\xi_1, \ldots, \xi_N)$ is a \textit{differential} operator, i.e., the action of the symmetric group (given by the operators $s_{ij}$) can be removed.

\subsection{Definition of the Bessel functions}\label{sec:besselsdef}

Given fixed $a_1, \ldots, a_N\in\C$, we can now consider the following \textit{system of hypergeometric differential equations}:
\begin{equation}\label{eqn:hypersystem}
P(\xi_1, \ldots, \xi_N) F(x) = P(a_1, \ldots, a_N) F(x), \ \textrm{ for all symmetric polynomial }P.
\end{equation}
A theorem of \cite{Op} gives the unique solution to the system of hypergeometric differential equations above.
The theorem (admittedly a simpler version of it) is the following:

\begin{thm}\label{thm:opdam}(\cite{Op})
For fixed and arbitrary $a_1, \ldots, a_N\in\C$, there exists a unique solution $F(x) = F(x_1, \ldots, x_N)$ to $(\ref{eqn:hypersystem})$, symmetric with respect to its variables, and normalized by
$$F(0, \ldots, 0) = 1.$$
We denote this solution by $\B_{(a_1, \ldots, a_N)}(x_1, \ldots, x_N; \theta)$.
The map $(a, x) \mapsto \B_{a}(x; \theta)$
admits an extension to a holomorphic function on $2N$ variables:
\begin{align*}
\C^N \times \C^N &\rightarrow \C\\
(a, x) &\mapsto \B_a(x; \theta).
\end{align*}
\end{thm}

Some remarks are in order:

\smallskip

1. From the uniqueness in Theorem $\ref{thm:opdam}$, it follows that $\B_{(a_1, \ldots, a_N)}(x_1, \ldots, x_N; \theta)$ is not only symmetric with respect to $x_1, \ldots, x_N$, but it is also symmetric with respect to the variables $a_1, \ldots, a_N$.

\smallskip

2. The considerations above actually define multivariate Bessel functions for more general values of $\theta$, including all points of the region $\{ \theta\in\C : \Re \theta \geq 0 \}$.
For instance, when $\theta = 0$, the Dunkl operators are $\xi_i = \partial/\partial x_i$, and the multivariate Bessel function is
$$\B_{(a_1, \ldots, a_N)}(x_1, \ldots, x_N; 0) = \frac{1}{N!}\sum_{\sigma\in S_N}{ \exp( a_1x_{\sigma(1)} + \ldots + a_Nx_{\sigma(N)} ) }.$$

3. When $\theta \in \{\frac{1}{2}, 1, 2\}$, the multivariate Bessel function $\B_{(a_1, \ldots, a_N)}(x_1, \ldots, x_N; \theta)$ can be identified with a spherical transform \cite{S}.
For example, when $\theta = 1$,
$$\B_{(a_1, \ldots, a_N)}(x_1, \ldots, x_N; 1) \sim \int_{U(N)}{\exp\left( \textrm{Tr}(U A U^* X) \right)  dU},$$
where $\sim$ means equality up to a constant of proportionality, $dU$ denotes the Haar measure on $U(N)$, and $A, X$ are diagonal matrices with diagonal entries $a_1, \ldots, a_N$ and $x_1, \ldots, x_N$, respectively.
This is already an indication of the non-trivial relation between Bessel functions and random matrices.
One can also obtain determinantal formulas for the spherical integrals by using the so-called Harish-Chandra-Itzykson-Zuber formula \cite{HC}, \cite{IZ}, and its relatives; see \cite{FIZ} and references therein.

\smallskip

4. The Dunkl operators and Bessel functions above are associated to the root system of type A.
The theory of Dunkl operators and multivariate Bessel functions has been developed for more general root systems.

\subsection{Properties of the Bessel functions}

The relation between the multivariate Bessel functions and the orbital beta processes is given by the following \textit{combinatorial formula}.

\begin{prop}\label{prop:combinatorial}
Let $y_1, \ldots, y_N\in\C$ be arbitrary, and $a = (a_1 > \dots > a_N) \in \PP_N$. Then
$$\B_a(y_1, \ldots, y_N; \theta) = \EE_{\{a_i^{(k)}\}} \left[
\exp\left( \sum_{k=1}^{N}
{y_k \left( \sum_{i=1}^{k}a^{(k)}_i - \sum_{j=1}^{k-1}a^{(k-1)}_j \right)} \right) \right],$$
where the expectation is taken with respect to the orbital beta process with top level $a = (a_1, \dots, a_N)$.
Equivalently,
\begin{multline*}
\B_a(y_1, \ldots, y_N; \theta) = \frac{\prod_{k=1}^{N}{\Gamma(k \theta)}}{\Gamma(\theta)^{\frac{N(N+1)}{2}}} \prod_{1 \leq i < j \leq N}{|a_i - a_j|^{1 - 2\theta}}\\
\times\int_{a^{(1)} \prec a^{(2)} \prec \dots \prec a^{(N)} = a}
\exp\left(  \sum_{k=1}^N y_k \left( \sum_{i=1}^k a^{(k)}_i - \sum_{j = 1}^{k-1} a^{(k-1)}_j \right) \right)\\
\prod_{k = 1}^{N - 1} \left( \prod_{1 \leq i < j \leq k} |a^{(k)}_i - a^{(k)}_j|^{2 - 2\theta}
\prod_{s = 1}^{k+1} \prod_{r = 1}^k |a_s^{(k+1)} - a_r^{(k)}|^{\theta - 1} \prod_{j=1}^k{da^{(k)}_j} \right).
\end{multline*}
\end{prop}

\medskip

The equivalence between the two formulas for the Bessel function $\B_a(y_1, \ldots, y_N; \theta)$ comes from the definition of the orbital beta process and Lemma $\ref{lem:normalization}$.

The proposition shows that $\B_a(y_1, \ldots, y_N; \theta)$ is a \textit{partial Fourier transform} for the orbital beta process with top level $a$.
In the form stated above, the proposition was taken from \cite{GoM}.
It had appeared before in \cite{GK}; see also \cite{FR}.
It is also a degeneration of the combinatorial formula for Jack polynomials, see \cite[VI.10.10--VI.10.11]{M} and Theorem $\ref{thm:okol}$ below.
When $y_1 = \ldots = y_N = 0$, the identity in the proposition is equivalent to the identity in Lemma $\ref{lem:normalization}$.

\begin{cor}\label{cor:bounded2}
If $a = (a_1 > \ldots > a_N)\in\PP_N$ and $y_1, \ldots, y_N\in\C$, then
$$|\B_a(y_1, \ldots, y_N; \theta)| \leq \B_a(\Re y_1, \ldots, \Re y_N; \theta)$$
\end{cor}
\begin{proof}
Using $|\exp(z)| = \exp(\Re z)$, we have
$$\left| \exp\left( \sum_{k=1}^{N} {y_k \left( \sum_{i=1}^{k}a^{(k)}_i - \sum_{j=1}^{k-1}a^{(k-1)}_j \right)} \right) \right| =
\exp\left( \sum_{k=1}^{N} {\Re y_k \left( \sum_{i=1}^{k}a^{(k)}_i - \sum_{j=1}^{k-1}a^{(k-1)}_j \right)} \right),$$
whenever $a_i^{(k)} \in \R$, for any $1 \leq i \leq k \leq N$.
The result then follows from Proposition $\ref{prop:combinatorial}$ and Jensen's inequality.
\end{proof}

\begin{cor}\label{cor:bounded}
If $a = (a_1 > \ldots > a_N)\in\PP_N$ and $y_1, \ldots, y_N\in\R$, then
$$|\B_a(\ii y_1, \ldots, \ii y_N; \theta)| \leq 1.$$
\end{cor}
\begin{proof}
This is a special case of Corollary $\ref{cor:bounded2}$; simply use the normalization $\B_a(0, \ldots, 0; \theta) = 1$.
\end{proof}

\begin{cor}\label{cor:additionlabel}
Let $a_1, \dots, a_N, y_1, \ldots, y_N, r$ be $2N+1$ arbitrary complex numbers.
Then
$$\B_{(a_1 + r, a_2 + r, \ldots, a_N + r)}(y_1, \ldots, y_N; \theta) = \exp(r(y_1 + \ldots + y_N))\B_{(a_1, a_2, \ldots, a_N)}(y_1, \ldots, y_N; \theta).$$
\end{cor}
\begin{proof}
When $(a_1 > \dots > a_N)\in\PP_N$ and $r\in\R$, the result follows from Proposition $\ref{prop:combinatorial}$.
Both sides of our desired identity are holomorphic on the $(N+1)$ variables $a_1, \ldots, a_N, r$, so the result extends by analytic continuation.
\end{proof}

\begin{cor}\label{cor:multiplicationlabel}
Let $a_1, \dots, a_N, y_1, \ldots, y_N, c$ be $2N+1$ arbitrary complex numbers.
Then
$$\B_{(ca_1, \ldots, ca_N)}(y_1, \ldots, y_N; \theta) = \B_{(a_1, \ldots, a_N)}(cy_1, \ldots, cy_N; \theta).$$
\end{cor}
\begin{proof}
Same proof as for Corollary $\ref{cor:additionlabel}$.
\end{proof}

\section{Observables for the orbital beta process and the Gaussian beta ensemble}\label{sec:observables}

Recall that $\mm^{a(N), \theta}$ is the beta orbital process with top level $a(N)\in\PP_N$.
Let $1 \leq m \leq N-1$ be an integer.
Denote by $\mm^{a(N), \theta}_m$ the pushforward of $\mm^{a(N), \theta}$ under the map
$$(a_i^{(k)})_{1 \leq i \leq k \leq N-1} \mapsto (a^{(m)}_j)_{1 \leq j \leq m},$$
i.e., under the map that sends an element of $\PP(a(N))$ to its $m$-th level.
Clearly the probability measure $\mm^{a(N), \theta}_m$ lives on the Weyl chamber $\W_m$.

\begin{prop}\label{prop:obsorbital}
For any $a(N)\in\PP_N$, complex numbers $y_1, \ldots, y_m\in\C$, and integers $N > m \geq 1$, we have
\begin{equation*}
\EE^{\mm^{a(N), \theta}_m} \left[ \B_{(a_1, \ldots, a_m)}(y_1, \ldots, y_m; \theta) \right]
= \B_{a(N)}(y_1, \ldots, y_m, 0^{N-m}; \theta),
\end{equation*}
where the expectation is taken with respect to the probability measure $\mm^{a(N), \theta}_m$ on the Weyl chamber $\mathcal{W}_m$, and $a = (a_1 > \ldots > a_m)\in \W_m$ is the random vector.
\end{prop}
\begin{proof}
It is convenient to denote $a(N)$ by $a^{(N)}$ sometimes.
Also denote the expectation $\EE^{\mm^{a(N), \theta}}$ with respect to $\mm^{a(N), \theta}$ simply by $\EE$.
We have the following chain of equalities, for any complex numbers $y_1, \dots, y_N$:
\begin{align*}
\B_{a(N)}&(y_1, \ldots, y_N; \theta) = \EE \left[ \exp\left( \sum_{k=1}^{N}{y_k ( |a^{(k)}| - |a^{(k-1)}| )} \right) \right]\\
&= \EE^{\mm_m^{a(N), \theta}} \left[\EE \left[
\exp\left( \sum_{k=1}^{N}{y_k ( |a^{(k)}| - |a^{(k-1)}| )} \right) \Bigg\vert a^{(m)}  \right] \right]\\
&= \EE^{\mm_m^{a(N), \theta}} \left[ \EE \left[
\exp\left( \sum_{k=1}^{m}{y_k ( |a^{(k)}| - |a^{(k-1)}| )} \right) \Bigg\vert a^{(m)}  \right]
\EE \left[ \exp\left( \sum_{k=m+1}^{N}{y_k ( |a^{(k)}| - |a^{(k-1)}| )} \right) \Bigg\vert a^{(m)}  \right] \right]\\
&= \EE^{\mm_m^{a(N), \theta}} \left[ \B_{a^{(m)}}(y_1, \ldots, y_m; \theta)\cdot
\EE \left[ \exp\left( \sum_{k=m+1}^{N}{y_k ( |a^{(k)}| - |a^{(k-1)}| )} \right) \Bigg\vert a^{(m)}  \right] \right].
\end{align*}
The first equality is simply Proposition \ref{prop:combinatorial}.
The second equality is a general property of conditional expectation.
The third equality follows from the fact that, under $\mm^{a(N), \theta}$, the distribution of $a^{(1)}, \ldots, a^{(m-1)}$ and the distribution of $a^{(m+1)}, \ldots, a^{(N-1)}$ are independent, if we condition on any value of $a^{(m)}\in\PP_m$.
The fourth and final equality is a consequence of Proposition \ref{prop:combinatorial} and the fact that, under $\mm^{a(N), \theta}$, if $a^{(m)}$ is fixed, then the joint distribution of $a^{(1)}, \ldots, a^{(m-1)}$ is given by $\mm^{a^{(m)}, \theta}$.
Finally, we obtain the desired result by setting $y_{m+1} = \dots = y_N = 0$
\end{proof}

\begin{prop}\label{prop:obsbessel}
For any $y_1, \ldots, y_m\in\C$, we have
\begin{equation*}
\EE^{\mm_m^{(\theta)}} \left[ \B_{(a_1, \ldots, a_m)}(y_1, \ldots, y_m; \theta) \right]
= \exp\left( \frac{1}{2\theta}\sum_{i=1}^m{y_i^2} \right),
\end{equation*}
where the expectation is taken with respect to the Gaussian beta ensemble $\mm_m^{(\theta)}$ of Definition $\ref{def:betagaussian}$ on the Weyl chamber $\mathcal{W}_m$, and $a = (a_1 > \ldots > a_m)\in \W_m$ is the random vector.
\end{prop}
\begin{proof}
This is the so-called Macdonald-Mehta integral, conjectured by Macdonald, \cite{M82}, and proved by Opdam, \cite{Op}, in greater generality (for more general root systems).
The statement is also a degeneration of Kadell's extension to Selberg's integral formula, see \cite{K} and Section $\ref{subsection:bessel}$.
\end{proof}

\section{Summary of asymptotic results for multivariate Bessel functions}\label{sec:formulas}

The explicit formulas for the study of asymptotics of multivariate Bessel functions are stated below.
They are degenerations of analogous formulas for \textit{normalized Jack polynomials} in \cite{Cu18b};
we defer their proofs to Section $\ref{sec:besselformulas}$.

First we need some terminology for the multivariate Bessel functions with all but a fixed number of variables set to zero: for integers $1 \leq m \leq N$ and any $a\in\C^N$, denote
$$\B_a(y_1, \ldots, y_m; N; \theta) := \B_a(y_1, \ldots, y_m, 0^{N-m}; \theta).$$

\begin{thm}\label{besselthm2}
Let $a = (a_1 > \ldots > a_N) \in \PP_N$ and $y\in\C$, $\Re y > 0$. Then the integral below converges absolutely and the identity holds
\begin{equation}\label{besselthm2eqn}
\B_a \left(y; N; \theta\right) = \frac{\Gamma(\theta N)}{y^{\theta N - 1}}\frac{1}{2\pi\ii}
\int_{(0^+)}^{-\infty}{\exp(yz) \prod_{i=1}^N{(z - a_i)^{-\theta}}dz},
\end{equation}
where the counter-clockwise oriented contour consists of the segment $[M + r\ii, \ M - r\ii]$ and horizontal lines $[M + r\ii, -\infty+r\ii)$, $[M - r\ii, -\infty - r\ii)$, for any fixed $M > a_1$, $r > 0$; see Figure $\ref{fig:Cminus}$.
\end{thm}

\begin{figure}
\begin{center}
\begin{tikzpicture}[decoration={markings,
mark=at position 1.5cm with {\arrow[line width=1pt]{>}},
mark=at position 3.2cm with {\arrow[line width=1pt]{>}},
mark=at position 5cm with {\arrow[line width=1pt]{>}},
mark=at position 6.5cm with {\arrow[line width=1pt]{>}},
mark=at position 8cm with {\arrow[line width=1pt]{>}},
mark=at position 9.8cm with {\arrow[line width=1pt]{>}},
mark=at position 11.7cm with {\arrow[line width=1pt]{>}}
}
]
\draw[help lines,->] (-4,0) -- (4,0) coordinate (xaxis);
\draw[help lines,->] (0,-1) -- (0,1) coordinate (yaxis);

\foreach \Point in {(-2, -0.03), (-1, -0.03), (1.2,-0.03)}{
    \node at \Point {\textbullet};
}
\node at (-2, 0.2) {$a_N$};
\node at (-1, 0.2) {$a_{N-1}$};
\node at (-0.2, 0.2) {$\cdots$};
\node at (0.5, 0.2) {$\cdots$};
\node at (1.2, 0.2) {$a_1$};

\path[draw,line width=1pt,postaction=decorate] (-4,-0.5) -- (2,-0.5) -- (2,0.5) -- (-4, 0.5);

\node[below] at (xaxis) {$\Re z$};
\node[left] at (yaxis) {$\Im z$};
\node[below left] {};
\end{tikzpicture}
\end{center}
\caption{Contour for Theorem $\ref{besselthm2}$}
\label{fig:Cminus}
\end{figure}

\begin{center}
\begin{figure}
\begin{center}
\begin{tikzpicture}[decoration={markings,
mark=at position 1.5cm with {\arrow[line width=1pt]{>}},
mark=at position 3.2cm with {\arrow[line width=1pt]{>}},
mark=at position 4.7cm with {\arrow[line width=1pt]{>}},
mark=at position 6.5cm with {\arrow[line width=1pt]{>}},
mark=at position 8.3cm with {\arrow[line width=1pt]{>}},
mark=at position 10cm with {\arrow[line width=1pt]{>}},
mark=at position 11.8cm with {\arrow[line width=1pt]{>}}
}
]

\draw[help lines,->] (-4,0) -- (4,0) coordinate (xaxis);
\draw[help lines,->] (0,-1) -- (0,1) coordinate (yaxis);

\path[draw,line width=1pt,postaction=decorate] (4,0.5) -- (-2,0.5) -- (-2,-0.5) -- (4,-0.5);

\foreach \Point in {(-1, -0.03), (0.2, -0.03), (3,-0.03)}{
    \node at \Point {\textbullet};
}
\node at (-1, 0.2) {$a_N$};
\node at (0.5, 0.2) {$a_{N-1}$};
\node at (2.3, 0.2) {$\cdots$};
\node at (1.5, 0.2) {$\cdots$};
\node at (3, 0.2) {$a_1$};

\node[below] at (xaxis) {$\Re z$};
\node[left] at (yaxis) {$\Im z$};
\node[below left] {};
\end{tikzpicture}
\end{center}
\caption{Contour for Theorem $\ref{besselthm1}$}
\label{fig:C}
\end{figure}
\end{center}

\begin{thm}\label{besselthm1}
Let $a = (a_1 > \ldots > a_N) \in \PP_N$ and $y\in\C$, $\Re y > 0$. Then the integral below converges absolutely and the identity holds
\begin{equation}\label{besselthm1eqn}
\B_a \left(-y; N; \theta\right) = -\frac{\Gamma(\theta N)}{y^{\theta N - 1}}\frac{1}{2\pi\ii}
\int_{(0^-)}^{+\infty}{\exp(-yz)\prod_{i=1}^N{(a_i - z)^{-\theta}}dz},
\end{equation}
where the counter-clockwise oriented contour consists of the segment $[M+r\ii, \ M-r\ii]$ and horizontal lines $[M+r\ii, +\infty+r\ii)$, $[M-r\ii, +\infty-r\ii)$, for any fixed $a_N > M$, $r > 0$; see Figure $\ref{fig:C}$.
\end{thm}

In both of the theorems above, the natural logarithm (needed to define $x^{-\theta}$ for complex values of $x$) is defined on $\C\setminus (-\infty, 0]$.
The branch of the logarithm does not allow us to close the contours and this is why we need infinite contours.

\begin{thm}\label{besselthm3}
Let $a \in \PP_N$ and let $m, N\in\N_+$ be such that $1\leq m\leq N-1$. Consider also any real numbers
\begin{equation*}
\begin{gathered}
y_1 > y_2 > \dots > y_m > y > 0,\\
\min_{i = 1, 2, \dots, m-1}{(y_i - y_{i+1})} > y.
\end{gathered}
\end{equation*}
Then
\begin{multline*}
\B_{a}({-y_1}, \ldots, {-y_m}; N; \theta)
\B_{a}({-y}; N; \theta) = \frac{\Gamma(N\theta)}{\Gamma((N-m)\theta)\Gamma(\theta)^m }
\frac{\prod_{1 \leq i < j \leq m }{(y_i - y_j)^{1 - 2\theta}}}{y^{m\theta} (y_1\cdots y_m)^{\theta}}\\
\times\int\dots\int { G_{\theta, y_1, \ldots, y_m, y}(z_1, \ldots, z_m) F_{y_1, \ldots, y_m, y}(z_1, \ldots, z_m)^{\theta(N-m)-1} }\\
{\B_{a}\left(- (y_1 + z_1), \ldots, -(y_m + z_m), -y+(z_1 + \dots + z_m); N; \theta\right)
\prod_{i=1}^m{(z_i^{\theta-1}dz_i)} },
\end{multline*}
where the domain of integration $\V_y$ in the integral is the compact subset of $\R^m$ defined by the inequalities
\begin{equation*}
\left\{
\begin{gathered}
z_1, \ldots, z_m \geq 0,\\
y \geq z_1 + \dots + z_m.
\end{gathered}
\right.
\end{equation*}

Moreover, the functions $G_{\theta, y_1, \ldots, y_m, y}(z_1, \ldots, z_m)$ and $F_{y_1, \ldots, y_m, y}(z_1, \ldots, z_m)$ are defined as follows (we denote $y_{m+1} := y - (z_1 + \dots + z_m)$ and $z_{m+1} := 0$ to simplify notation):
\begin{equation}\label{defn:FGfns}
\begin{aligned}
G_{\theta, y_1, \ldots, y_m, y}(z_1, \ldots, z_m) &:=
\prod_{1 \leq i < j \leq m}{(y_i - y_j + z_i)^{\theta - 1}}\prod_{1 \leq i < j \leq m+1}{(y_i - y_j + z_i - z_j)(y_i - y_j - z_j)^{\theta - 1}},\\
F_{y_1, \ldots, y_m, y}(z_1, \ldots, z_m) &:= (1 - (z_1 + \dots + z_m)/y)\prod_{i=1}^m{(1 + z_i/y_i)}.
\end{aligned}
\end{equation}
\end{thm}

The three theorems above will be used to prove the following limit for multivariate Bessel functions when all but a fixed number of arguments are set to zero.
The proof is in Section $\ref{sec:asymptotics}$.

\begin{thm}\label{thm:asymptoticbessels}
Let $m\in\N_+$ and $\{ a(N)\in\PP_N \}_{N \geq 1}$ be a regular sequence (in the sense of Definition \ref{def:regularsignatures}) of ordered tuples with limiting measure $\mu$. Then
$$
\lim_{N\rightarrow\infty}{\B_{a(N)}\left( \frac{y_1}{\sqrt{N}}, \dots, \frac{y_m}{\sqrt{N}}; N; \theta \right)}
\exp\left( -\sqrt{N} \cdot \E[\mu_N] \sum_{i=1}^m{y_i} \right)
= \exp\left( \frac{\Var[\mu]}{2\theta}\sum_{i=1}^m{y_i^2} \right),
$$
uniformly for $y_1, \ldots, y_m$ belonging to compact subsets of $\C$.

As usual, $\mu_N := \frac{1}{N}\sum_{i=1}^N{\delta_{a(N)_i/N}}$ and the notations $\E[\nu], \Var[\nu]$ were defined in \eqref{eqn:meanvariance}.
\end{thm}

\section{Proof of the main theorem}

We prove Theorem \ref{thm:GBE}, assuming the validity of Theorem \ref{thm:asymptoticbessels}.
In fact, we shall see that the uniform convergence of Theorem \ref{thm:asymptoticbessels} is stronger than we need: pointwise convergence would be enough.

\subsection{The Dunkl transform and the Inversion theorem}

In this subsection, let $m \geq 1$ be fixed; we consider spaces of functions on $m$ variables.

We work with the ``symmetrized'' version of the Dunkl transform in \cite{dJ}.
All spaces of functions on $\R^m$ have a \textit{symmetric} analogue.
For instance, we denote by $\SSS^{\sym}$ the space of Schwarz functions on $\R^m$ which are symmetric with respect to its $m$ arguments.
Similarly, denote by $(C_0^{\infty})^{\sym}$ the space of smooth, compactly supported and symmetric functions on $\R^m$.

Let $w_{\theta} : \R^m \rightarrow [0, \infty)$ be the weight function
$$w_{\theta}(x_1, \ldots, x_m) := \prod_{1 \leq i < j \leq m}|x_i - x_j|^{2\theta}.$$

We use both $x = (x_1, \ldots, x_m)$ and $\lambda = (\lambda_1, \ldots, \lambda_m)$ to denote vectors in $\R^m$.

For a symmetric function $f$ on $\R^m$ such that $f\in L^1(\R^m, |w_{\theta}(x)|dx)$, define the action of the operators $D_{\theta}, E_{\theta}$ on $f$ by
\begin{align*}
(D_{\theta}f)(\lambda) &:= \int_{\R^m}{f(x)\B_{-\ii\lambda}(x; \theta) w_{\theta}(x)dx}, \ \lambda\in\R^m,\\
(E_{\theta}f)(x) &:= \int_{\R^m}{f(\lambda)\B_{\ii\lambda}(x; \theta) w_{\theta}(\lambda)d\lambda}, \ x\in\R^m,
\end{align*}
where $\pm\ii\lambda := (\pm \ii\lambda_1, \ldots, \pm \ii\lambda_m)$.
Because of Corollary $\ref{cor:bounded}$, the integrals above are absolutely convergent, so $D_{\theta}f, E_{\theta}f$ are well-defined.
Observe also that the functions $D_{\theta}f, E_{\theta}f$ on $\R^m$ are symmetric.

Each of the operators $D_{\theta}, E_{\theta}$ is called the \textit{Dunkl transform}.
When $\theta = 0$, they are simply (constant multiples of) the Fourier transform and its inverse.
It is proved in \cite{dJ} that if $f\in\SSS^{\sym}$, then $D_{\theta}f, E_{\theta}f\in\SSS^{\sym}$.
Even more is true, as the next theorem shows.

\begin{thm}[Inversion Theorem \cite{dJ}]\label{thm:dJ}
Both $D_{\theta}, E_{\theta}$ are homeomorphisms of $\SSS^{\sym}$ and $E_{\theta}\circ D_{\theta} = D_{\theta}\circ E_{\theta} = c_{\theta}^2 \mathbf{1}_{\SSS^{\sym}}$, where the constant is
$$c_{\theta} = (2\pi)^{\frac{m}{2}} m! \prod_{j=1}^m{\frac{\Gamma(j\theta)}{\Gamma(\theta)}}.$$
\end{thm}

\subsection{Proof of the main theorem}
We begin with the following lemma.

\begin{lem}\label{reductionschwarz}
Let $\{\mm_N\}_{N\geq 1}$, $\mm$ be probability distributions on the Weyl chamber $\W_m$ such that
\begin{equation}\label{schwarz}
\EE^{\mm_N}[f(x_1, \ldots, x_m)] \xrightarrow{N \rightarrow \infty} \EE^{\mm}[f(x_1, \ldots, x_m)]
\end{equation}
for all functions $f$ on $\W_m$ which admit an extension $F$ to $\R^m$ such that $F\in(C_{0}^{\infty})^{\sym}$.
Then $\mm_N \rightarrow \mm$ weakly.
\end{lem}
\begin{proof}
Let $g$ be any bounded, continuous function on $\W_m$; say $\sup_{x\in\W_m}|g(x)| < M$. 
Clearly $g$ is the restriction of a bounded, continuous and symmetric function $G$ on $\R^m$; also $\sup_{x\in\R^m}{|G(x)|} < M$.

Let $\epsilon > 0$ be arbitrary. For any $r>0$, let $B(0, r)$ (resp. $B[0, r]$) be the open (resp. closed) ball of radius $r$ in $\R^m$.
Let $R>0$ be large enough so that
$$\mm(\W_m \setminus B(0, R)) < \epsilon/M.$$
Then for large enough $N$, we also have
$$\mm_N(\W_m \setminus B(0, R)) < \epsilon/M.$$
(Use the hypothesis $(\ref{schwarz})$ for nonnegative functions in $(C_0^{\infty})^{\sym}$ which are $1$ in $B(0, R)$ and are $0$ outside a slightly larger ball.)

Find a function $F\in C_0^{\infty}$ on $\R^m$ such that $\sup_{x\in B[0, R]}|F(x) - G(x)| < \epsilon$ and $\sup_{x\in\R^m}{|F(x)|} < 2M$.
By symmetrizing, we can actually find $F\in(C_0^{\infty})^{\sym}$ with these properties.
Thus its restriction $f$ to $\W_m$ satisfies $(\ref{schwarz})$, by assumption.
Moreover it satisfies $\sup_{x\in\W_m}{|f(x)|} < 2M$ and $\sup_{x\in B[0, R]\cap\W_m}{|f(x) - g(x)|} < \epsilon$.
We can bound:
$$
|\EE^{\mm_N}[g] - \EE^{\mm}[g]| \leq |\EE^{\mm_N}[g - f]| + |\EE^{\mm}[g - f]| + |\EE^{\mm_N}[f] - \EE^{\mm}[f]|.
$$
Also,
\begin{align*}
|\EE^{\mm_N}[g - f]| &\leq |\EE^{\mm_N}[(g - f)\mathbf{1}_{B[0, R]}]| + (M+2M)\mm_N(\W_m \setminus B(0, R))
< \epsilon + 3M \cdot \frac{\epsilon}{M} = 4\epsilon;\\
|\EE^{\mm}[g - f]| &\leq |\EE^{\mm}[(g - f)\mathbf{1}_{B[0, R]}]| + (M+2M)\mm(\W_m \setminus B(0, R))
< \epsilon + 3M \cdot \frac{\epsilon}{M} = 4\epsilon.
\end{align*}
It follows that
$$
|\EE^{\mm_N}[g] - \EE^{\mm}[g]| \leq 8\epsilon  + |\EE^{\mm_N}[f] - \EE^{\mm}[f]|.
$$
Since $\epsilon > 0$ was arbitrary, this finally shows that $\EE^{\mm_N}[g] \xrightarrow{N \rightarrow \infty} \EE^{\mm}[g]$, and we are done.
\end{proof}

\smallskip

\begin{proof}[Proof of Theorem $\ref{thm:GBE}$]

\textit{Step 1.}
Let $m \geq 1$ be an integer and let $\{a(N) \in \PP_N\}_{N \geq 1}$ be a regular sequence of ordered tuples.
Let $(a_i^{(k, N)})_{1 \leq i \leq k \leq N-1}$ be a vector distributed according to the orbital beta corners process $\mm^{a(N), \theta}$ with top level $a(N)$.
Also let $\mm^{a(N), \theta}_m$ be the pushforward of the distribution of $(a_i^{(k, N)})_{1 \leq i \leq k \leq N-1}$ under the map
$$(a_i^{(k, N)})_{1 \leq i \leq k \leq N-1} \mapsto (a^{(m, N)}_j)_{1 \leq j \leq m}.$$

Denote by $\mm_N$ the pushforward of $\mm_m^{a(N), \theta}$ under the map
$$(x_1, \ldots, x_m) \mapsto \left( \frac{x_1 - N \E[\mu_N]}{\sqrt{N \Var[\mu]}}, \ldots, \frac{x_m - N \E[\mu_N]}{\sqrt{N \Var[\mu]}} \right)$$
and denote by $\mm$ the Gaussian beta ensemble of Definition $\ref{def:betagaussian}$.
We shall prove the weak convergence $\mm_N \xrightarrow{N\rightarrow\infty} \mm$, which is the statement of Theorem $\ref{thm:GBE}$.
Because of Lemma $\ref{reductionschwarz}$, it suffices to prove
\begin{equation}\label{conv.schwarz}
\EE^{\mm_N}[f(x_1, \ldots, x_m)] \xrightarrow{N \rightarrow \infty} \EE^{\mm}[f(x_1, \ldots, x_m)]
\end{equation}
for functions $f$ on $\W_m$ that admit extensions to $\R^m$ that belong to $(C_0^{\infty})^{\sym}$.
We will in fact prove $(\ref{conv.schwarz})$ more generally for all functions $f$ that admit extensions to $\R^m$ that belong to $\SSS^{\sym}$.

\smallskip

\textit{Step 2.}
Proposition $\ref{prop:obsorbital}$ shows
$$
\EE^{\mm^{a(N), \theta}_m} \left[ \B_{(a_1, \ldots, a_m)}(y_1, \ldots, y_m; \theta) \right]
= \B_{a(N)}(y_1, \ldots, y_m, 0^{N-m}; \theta),
$$
for any $y_1, \ldots, y_m\in\C$.
On the other hand, Corollaries $\ref{cor:additionlabel}$ and $\ref{cor:multiplicationlabel}$ show
\begin{multline*}
\B_{\left( \frac{a_1 - N \E[\mu_N]}{\sqrt{N \Var[\mu]}}, \ldots, \frac{a_m - N \E[\mu_N]}{\sqrt{N \Var[\mu]}} \right)}(y_1, \ldots, y_m; \theta)\\
= \B_{(a_1, \ldots, a_m)}\left( \frac{y_1}{\sqrt{N \Var[\mu]}}, \ldots, \frac{y_m}{\sqrt{N \Var[\mu]}} \right)
\exp\left( -\frac{\sqrt{N}\E[\mu_N]}{\sqrt{\Var[\mu]}}(y_1 + \ldots + y_m) \right).
\end{multline*}
Combining the last two displays, and by our definition of $\mm_N$, we have
\begin{multline*}
\EE^{\mm_N}\left[ \B_{(a_1, \ldots, a_m)}(y_1, \ldots, y_m; \theta) \right] = \EE^{\mm^{a(N), \theta}_m}\left[ \B_{\left( \frac{a_1 - N \E[\mu_N]}{\sqrt{N \Var[\mu]}}, \ldots, \frac{a_m - N \E[\mu_N]}{\sqrt{N \Var[\mu]}} \right)}(y_1, \ldots, y_m; \theta) \right]\\
= \B_{a(N)}\left( \frac{y_1}{\sqrt{N \Var[\mu]}}, \ldots, \frac{y_m}{\sqrt{N \Var[\mu]}}, 0^{N-m}; \theta \right)
\exp\left( -\frac{\sqrt{N}\E[\mu_N]}{\sqrt{\Var[\mu]}}(y_1 + \ldots + y_m) \right).
\end{multline*}
Proposition $\ref{prop:obsbessel}$ gives the analogous observable for $\mm$:
$$\EE^{\mm}[ \B_{(a_1, \ldots, a_m)}(y_1, \ldots, y_m; \theta) ] = \exp\left( \frac{1}{2\theta} \sum_{i=1}^m{y_i^2} \right).$$
Then, by taking into account the last two displays, Theorem $\ref{thm:asymptoticbessels}$ implies the limit
\begin{equation}\label{conv.bessel}
\EE^{\mm_N}[\B_{(a_1, \ldots, a_m)}(y_1, \ldots, y_m; \theta)] \xrightarrow{N \rightarrow \infty} \EE^{\mm}[\B_{(a_1, \ldots, a_m)}(y_1, \ldots, y_m; \theta)].
\end{equation}
Note that $(\ref{conv.bessel})$ is exactly the desired $(\ref{conv.schwarz})$, but for $f(a_1, \ldots, a_m) = \B_{(a_1, \ldots, a_m)}(y_1, \ldots, y_m; \theta)$.
It remains to extend it for all $f\in\SSS^{\sym}$.

\smallskip

\textit{Step 3.}
Let $f\in\SSS^{\sym}$, so that $f = c_{\theta}^{-2} \cdot D_{\theta}(E_{\theta}f)$ by Theorem $\ref{thm:dJ}$.
This means
$$f(a) = c_{\theta}^{-2} \int_{\R^m}{(E_{\theta}f)(x)\B_{-\ii a}(x; \theta) w_{\theta}(x)dx},$$
where $a = (a_1, \ldots, a_m)\in\R^m$ and we denote $-\ii a := (-\ii a_1, \ldots, -\ii a_m)$.
Corollary $\ref{cor:multiplicationlabel}$ shows $\B_{-\ii a}(x; \theta) = \B_{a}(-\ii x; \theta)$, therefore
\begin{equation}\label{inversioneqn}
f(a) = c_{\theta}^{-2} \int_{\R^m}{(E_{\theta}f)(x)\B_{a}(-\ii x; \theta) w_{\theta}(x)dx}.
\end{equation}
By virtue of Corollary $\ref{cor:bounded}$, the modulus of the integrand in $(\ref{inversioneqn})$ is upper bounded by $C(1 + |x_1|)^{2m\theta}\cdots (1 + |x_m|)^{2m\theta} |E_{\theta}f(x)|$, for some constant $C>0$.
Since $E_{\theta}f \in \SSS^{\sym}$, it follows that the integral in $(\ref{inversioneqn})$ converges absolutely, and uniformly on $a$.
As a result, we can apply Fubini's theorem and the dominated convergence theorem as follows:
\begin{align*}
\EE^{\mm_N}[f(a)] &= c_{\theta}^{-2} \cdot \EE^{\mm_N}\left[ \int_{\R^m}{(E_{\theta}f)(x)\B_{a}(-\ii x; \theta) w_{\theta}(x)dx} \right]\\
&= c_{\theta}^{-2} \cdot \int_{\R^m}{(E_{\theta}f)(x) \EE^{\mm_N}[ \B_{a}(-\ii x; \theta) ] w_{\theta}(x)dx}\\
&\rightarrow c_{\theta}^{-2} \cdot \int_{\R^m}{(E_{\theta}f)(x) \EE^{\mm}[ \B_{a}(-\ii x; \theta) ] w_{\theta}(x)dx}\\
&= c_{\theta}^{-2} \cdot \EE^{\mm} \left[ \int_{\R^m}{(E_{\theta}f)(x) \B_{a}(-\ii x; \theta) w_{\theta}(x)dx} \right]\\
&= \EE^{\mm} [f(a)].
\end{align*}
The last equation above uses $(\ref{inversioneqn})$ again.
We are done.
\end{proof}

\section{Proof of the formulas in Section $\ref{sec:formulas}$}\label{sec:besselformulas}

\subsection{Formulas for normalized Jack polynomials}

Given a weakly decreasing sequence of integers $\lambda = (\lambda_1 \geq \dots \geq \lambda_N) \in \Z^N$, the Jack polynomial $J_{\lambda}(x_1, \ldots, x_N; \theta)$ is a symmetric, homogeneous Laurent polynomial of degree $|\lambda|$ with coefficients being rational functions of $\theta$.
We refer the reader to \cite[Chapter VI.10]{M} or \cite{St} for details on their definition, and also to \cite[Section 2]{Cu18b} for a short summary of their properties.

For two integers $1 \leq m \leq N$, a sequence $\lambda = (\lambda_1 \geq \dots \geq \lambda_N)\in\Z^N$, and the variables $x_1, \ldots, x_m$, let us have a special notation for a certain normalization of Jack polynomials:
$$
J_{\lambda}(x_1, \ldots, x_m; N; \theta) := \frac{J_{\lambda}(x_1, \ldots, x_m, 1^{N-m}; \theta)}{J_{\lambda}(1^N; \theta)}.
$$
The denominator $J_{\lambda}(1^N; \theta)$ is nonzero whenever $\theta > 0$, see \cite[VI.10]{M}.

\begin{thm}[\cite{Cu18b}, Theorem 2.8]
Let $\lambda = (\lambda_1 \geq \dots \geq \lambda_N) \in \Z^N$ and $y\in\C$, $\Re y > 0$. Then the integral below converges absolutely and the identity holds:
$$
J_{\lambda}\left(e^y; N; \theta\right) = \frac{\Gamma(\theta N)}{(1 - e^{-y})^{\theta N - 1}}\frac{1}{2\pi\ii}\int^{-\infty}_{(0^+)}{\exp(yz)\prod_{i=1}^N{\frac{\Gamma(z - (\lambda_i -\theta i + \theta))}{\Gamma(z - (\lambda_i - \theta i) )}}dz}.
$$
The contour above is counter-clockwise oriented and consists of the segment $[M + r\ii, \ M - r\ii]$ and horizontal lines $[M + r\ii, -\infty+r\ii)$, $[M - r\ii, -\infty - r\ii)$, for any fixed $M > \lambda_1$, $r > 0$.
\end{thm}

\begin{thm}[\cite{Cu18b}, Theorem 2.7]\label{jackthm1}
Let $\lambda = (\lambda_1 \geq \dots \geq \lambda_N) \in \Z^N$ and $y\in\C$, $\Re y > 0$. Then the integral below converges absolutely and the identity holds:
\begin{equation*}
J_{\lambda}\left(e^{-y}; N; \theta\right) = -\frac{\Gamma(\theta N)}{(1 - e^{-y})^{\theta N - 1}}\frac{1}{2\pi\ii}\int^{+\infty}_{(0^-)}{\exp(-yz)\prod_{i=1}^N{\frac{\Gamma(\lambda_i + \theta(N-i)-z)}{\Gamma(\lambda_i + \theta(N-i+1)-z)}}dz}.
\end{equation*}
The contour above is counter-clockwise oriented and consists of the segment $[M+r\ii, \ M-r\ii]$ and horizontal lines $[M+r\ii, +\infty+r\ii)$, $[M-r\ii, +\infty-r\ii)$, for any fixed $\lambda_N > M$, $r > 0$.
\end{thm}

\begin{thm}[\cite{Cu18b}, Theorem 3.1]\label{jackthm2}
Let $\lambda = (\lambda_1 \geq \dots \geq \lambda_N) \in \Z^N$, and let $m, N\in\N_+$ be such that $1\leq m\leq N-1$. Consider also any real numbers $y_1, \dots, y_m, y$ satisfying
\begin{equation*}
\begin{gathered}
y_1 > y_2 > \dots > y_m > y > 0,\\
\min_{i = 1, 2, \dots, m-1}{(y_i - y_{i+1})} > y.
\end{gathered}
\end{equation*}
Then
\begin{multline}\label{eqn:pierijack}
J_{\lambda}(e^{-y_1}, \ldots, e^{-y_m}; N; \theta)
J_{\lambda}(e^{-y}; N; \theta) = \frac{\Gamma(N\theta)}
{\Gamma((N - m)\theta)\Gamma(\theta)^{m}}
\frac{\prod_{1\leq i < j\leq m}{(1 - e^{-y_i + y_j})^{1 - 2\theta}}}{e^{ym\theta}(1 - e^{-y})^{m\theta}}\\
\times \frac{1}{\prod_{i=1}^m{(1 - e^{-y_i})^{\theta}}} \int\dots\int { G^{\trig}_{\theta, y_1, \dots, y_m, y}(z_1, \ldots, z_m) F^{\trig}_{y_1, \ldots, y_m, y}(z_1, \ldots, z_m)^{\theta(N-m)-1} }\\
J_{\lambda}\left(e^{- (y_1 + z_1)}, \ldots, e^{-(y_m + z_m)}, e^{-y+(z_1 + \dots + z_m)}; N; \theta\right) \prod_{i=1}^m{\left( (1 - e^{-z_i})^{\theta-1}dz_i \right)},
\end{multline}
where the functions $G^{\trig}_{\theta, y_1, \ldots, y_m, y}$ and $F^{\trig}_{y_1, \ldots, y_m, y}$ are (denote $z_{m+1} := y - (z_1 + \dots + z_m)$ and $y_{m+1} := 0$ to simplify notation):
\begin{equation}\label{defn:FGfnstrig}
\begin{aligned}
G^{\trig}_{\theta, y_1, \ldots, y_m, y}(z_1, \ldots, z_m) :=& \exp\left( {\theta (mz_1 + (m-1)z_2 \cdots + z_m)} \right)
\prod_{1\leq i < j\leq m}{\left(1 - e^{-y_i + y_j - z_i}\right)^{\theta-1}}\\
&\times\prod_{1\leq i<j\leq m+1}{(1 - e^{-y_i - z_i + y_j +z_j}) (1 - e^{-y_i + y_j + z_j})^{\theta-1} },\\
F_{y_1, \ldots, y_m, y}^{\trig}(z_1, \ldots, z_m) :=& (1 - e^{-y})^{-1} \prod_{j=1}^m{(1 - e^{-y_j})^{-1}} \prod_{i=1}^{m+1}{(1 - e^{-y_i - z_i})},
\end{aligned}
\end{equation}
and the domain of integration $\V_y$ in the integral of $(\ref{eqn:pierijack})$ is the compact subset of $\R^m$ defined by the inequalities
\begin{equation*}
\left\{
\begin{gathered}
z_1, \ldots, z_m \geq 0,\\
y \geq z_1 + \dots + z_m.
\end{gathered}
\right.
\end{equation*}
\end{thm}

\begin{rem}
The three theorems above are mild reformulations of the corresponding theorems in \cite{Cu18b}: the arguments $x, x_i$ in that paper have been replaced by $e^{y}, e^{y_i}$ (or $e^{-y}, e^{-y_i}$), and also the change of variables $w_i = e^{-z_i}$ was needed to obtain the integral $(\ref{eqn:pierijack})$ (the corresponding theorem in \cite{Cu18b} involved variables $w_1, \ldots, w_m$).
\end{rem}

\subsection{Limit from Jack polynomials to Bessel functions}\label{subsection:bessel}

The multivariate Bessel functions are degenerations of Jack polynomials in the following limit regime.

\begin{thm}[\cite{OkOl}, Section 4]\label{thm:okol}
If $a = (a_1 > \dots > a_N)\in\PP_N$ and $x = (x_1, \dots, x_N) \in \C^N$, then
\begin{equation*}
\lim_{\epsilon \rightarrow 0^+} \frac{J_{\lfloor \epsilon^{-1}a \rfloor}(e^{\epsilon x}; \theta)}{J_{\lfloor \epsilon^{-1}a \rfloor}(1^N; \theta)} = \B_a(x; \theta),
\end{equation*}
where we denoted $\lfloor\epsilon^{-1}a \rfloor := (\lfloor \epsilon^{-1}a_1 \rfloor, \dots, \lfloor \epsilon^{-1}a_N \rfloor)$ and $e^{\epsilon x} := (e^{\epsilon x_1}, \dots, e^{\epsilon x_N})$.
The limit is uniform for $x$ belonging to compact subsets of $\C^N$.
\end{thm}

\subsection{Proofs of Theorems $\ref{besselthm2}$--$\ref{besselthm3}$}

The proofs of Theorems $\ref{besselthm2}$ and $\ref{besselthm1}$ are almost identical, therefore we shall omit the proof of Theorem $\ref{besselthm2}$, but give the proof of Theorem $\ref{besselthm1}$.

\begin{proof}[Proof of Theorem $\ref{besselthm1}$]

In the formula of Theorem $\ref{jackthm1}$, set $\lambda_i = \lfloor \epsilon^{-1}a_i \rfloor$, $i = 1, 2, \ldots, N$, $y \mapsto \epsilon y$, and make the change of variables $z \mapsto \epsilon^{-1}z$; the result is:
\begin{equation}\label{formulaeps}
J_{\lambda}\left(e^{-\epsilon y}; N; \theta\right) =
-\frac{\Gamma(\theta N)}{(1 - e^{-\epsilon y})^{\theta N - 1}}\frac{\epsilon^{-1}}{2\pi\ii}\int^{+\infty}_{(0^-)}{\exp(-yz)\prod_{i=1}^N{\frac{\Gamma(\lfloor \epsilon^{-1}a_i \rfloor + \theta(N-i) - \epsilon^{-1}z)}{\Gamma(\lfloor \epsilon^{-1}a_i \rfloor + \theta(N-i+1) - \epsilon^{-1}z)}}dz}.
\end{equation}
The contour in the integral above is of the same form as the desired one in Theorem $\ref{besselthm1}$.
The only important point is that the contour does not depend on $\epsilon$, as we want to take limits when this variable tends to zero.

A well-known result about Gamma functions is the uniform limit
$$\frac{\Gamma(x+a)}{\Gamma(x + b)} = x^{a - b} (1 + O(1/x)), \ |x| \rightarrow \infty,$$
for $x$ avoiding the points $-a, -a-1, \dots$ and $-b, -b-1, \dots$, e.g. see \cite{TE}.
As a result, for any $i = 1, 2, \ldots, N$, we have the limit
\begin{equation}\label{limitgamma1}
\lim_{\epsilon \rightarrow 0^+}{ \epsilon^{-\theta}
\frac{\Gamma(\lfloor \epsilon^{-1}a_i \rfloor + \theta(N-i) - \epsilon^{-1}z)}{\Gamma(\lfloor \epsilon^{-1}a_i \rfloor + \theta(N-i+1) - \epsilon^{-1}z)} } = (a_i - z)^{-\theta}.
\end{equation}
Moreover,
$$\lim_{\epsilon \rightarrow 0^+}{ \frac{\epsilon^{\theta N - 1}}{(1 - e^{-\epsilon y})^{\theta N - 1}} } = \frac{1}{y^{\theta N - 1}}.$$
The previous limits show that the integrand in the right hand side of $(\ref{formulaeps})$ converges to the integrand in the right hand side of $(\ref{besselthm1eqn})$, as $\epsilon \rightarrow 0^+$.
Moreover, Theorem $\ref{thm:okol}$ shows also that the left hand side of $(\ref{formulaeps})$ converges to the left hand side of $(\ref{besselthm1eqn})$, as $\epsilon \rightarrow 0^+$.
Therefore it remains to apply the dominated convergence theorem and show that the integrands in $(\ref{formulaeps})$ are uniformly bounded when $\epsilon$ is a small positive real number, and when $z$ is in the contour with $|z|$  large.
But this is a consequence of the fact that $\exp(-yz)$ is uniformly bounded for $z$ in the contour, that the limit $(\ref{limitgamma1})$ is uniform for $z$ in the contour, and that the modulus of
$$\prod_{i=1}^N{ (a_i - z)^{-\theta} }$$
goes to zero as $|z|$ tends to infinity.
\end{proof}

In the proof below, we use bold letters to denote tuples: $\bfy := (y_1, \ldots, y_m, y)$, $\bfz := (z_1, \ldots, z_m)$, $\epsilon\bfy := (\epsilon y_1, \ldots, \epsilon y_m, \epsilon y)$, etc.

\begin{proof}[Proof of Theorem $\ref{besselthm3}$]
In the formula of Theorem $\ref{jackthm2}$, set $\lambda_i = \lfloor \epsilon^{-1}a_i \rfloor$, $i = 1, \ldots, N$; $y_j \mapsto \epsilon y_j$, $j = 1, \ldots, m$; $y \mapsto \epsilon y$; finally make the change of variables $z_k \mapsto \epsilon z_k$, $k = 1, 2, \ldots, m$:
\begin{multline}\label{eqn:pieri2}
J_{\lfloor \epsilon^{-1}a \rfloor}(e^{-\epsilon y_1}, \ldots, e^{-\epsilon y_m}; N; \theta)
J_{\lfloor \epsilon^{-1}a \rfloor}(e^{-\epsilon y}; N; \theta) =\\
\frac{\Gamma(N\theta)}{\Gamma((N-m)\theta)\Gamma(\theta)^{m}}
\frac{\prod_{1\leq i < j\leq m}{(1 - e^{-\epsilon(y_i - y_j)})^{1 - 2\theta}}}{e^{\epsilon ym\theta}(1 - e^{-\epsilon y})^{m\theta}\prod_{i=1}^m{(1 - e^{-\epsilon y_i})^{\theta}}}\\
\times \int\dots\int { G^{\trig}_{\theta, \epsilon \bfy}(\epsilon \bfz) F^{\trig}_{\epsilon \bfy}(\epsilon \bfz)^{\theta(N-m)-1} }\\
J_{\lfloor \epsilon^{-1}a \rfloor}\left(e^{- \epsilon(y_1 + z_1)}, \ldots, e^{- \epsilon(y_m + z_m)}, e^{-\epsilon(y-(z_1 + \dots + z_m))}; N; \theta\right) \prod_{i=1}^m{\left( (1 - e^{-\epsilon z_i})^{\theta-1} \epsilon dz_i \right)},
\end{multline}
The resulting domain of integration $\V_y$ remains exactly the same after the re-scaling of variables $y_j$'s and $z_k$'s; in particular, it does not depend on $\epsilon$.
We only have to take the limit of the resulting equation as $\epsilon \rightarrow 0^+$.

Theorem $\ref{thm:okol}$ shows that the limit of the left hand side of $(\ref{eqn:pieri2})$, as $\epsilon \rightarrow 0^+$, equals
$$\B_a(-y_1, \ldots, -y_m; N; \theta)\B_a(-y; N; \theta).$$

Next we look at the right hand side.
We frequently use (without further mention):
$$1 - e^{-\epsilon A} = \epsilon A (1 + O (\epsilon)), \ e^{\epsilon B} = 1 + O(\epsilon),$$
for any fixed $A, B\in\C$.
The second line of $(\ref{eqn:pieri2})$ equals
$$\frac{\Gamma(N\theta)}{\Gamma((N-m)\theta)\Gamma(\theta)^{m}}
\cdot\epsilon^{{m \choose 2}(1 - 2\theta) - 2m\theta}
\cdot\frac{\prod_{1 \leq i < j \leq m}(y_i - y_j)^{1 - 2\theta}}{y^{m\theta}(y_1\cdots y_m)^{\theta}}(1+O(\epsilon)).$$

Now from $(\ref{defn:FGfns})$ and $(\ref{defn:FGfnstrig})$, we obtain
\begin{align*}
G^{\textrm{trig}}_{\theta, \epsilon\bfy}(\epsilon\bfz) &= \epsilon^{{m \choose 2}(\theta-1) + {m+1 \choose 2}\theta}\cdot G_{\theta, \bfy}(\bfz) (1 + O(\epsilon));\\
F^{\textrm{trig}}_{\epsilon\bfy}(\epsilon\bfz) &= F_{\bfy}(\bfz)(1 + O(\epsilon)).
\end{align*}
The third line of $(\ref{eqn:pieri2})$ becomes
$$\epsilon^{{m \choose 2}(\theta-1) + {m+1 \choose 2}\theta} \cdot G_{\theta, \bfy}(\bfz)F_{\bfy}(\bfz)(1 + O(\epsilon)).$$
Finally, from Theorem $\ref{thm:okol}$, the fourth line of $(\ref{eqn:pieri2})$ is equal to
$$
\B_a(-(y_1+z_1), \ldots, -(y_m+z_m), -(y-z_1-\cdots-z_m); N; \theta)
\cdot \epsilon^{m\theta} \cdot \prod_{i=1}^m{z_i^{\theta - 1}} (1 + o(1)).
$$
Combining the estimates above, we obtain Theorem $\ref{besselthm3}$.
\end{proof}

\section{Proof of the asymptotic theorem in Section $\ref{sec:formulas}$}\label{sec:asymptotics}

In this section, we prove Theorem $\ref{thm:asymptoticbessels}$.

\subsection{A preliminary lemma}

\begin{lem}\label{lem:meanzero}
If Theorem $\ref{thm:asymptoticbessels}$ holds in the special case that the regular sequence $\{ a(N) \}_{N \geq 1}$ satisfies $\sum_{i=1}^N{a(N)_i} = 0$ for all $N\in\N_+$, then it also holds in full generality.
\end{lem}
\begin{proof}
Let $m\in\N_+$ and $\{a(N)\in\PP_N\}_{N \geq 1}$ be a regular sequence of ordered tuples with a limiting meaure $\mu$, satisfying the conditions in the statement of Theorem $\ref{thm:asymptoticbessels}$.
Let
$$r_N := \frac{a(N)_1 + \dots + a(N)_N}{N}, \quad N \geq 1.$$
If we let $\mu_N := \frac{1}{N}\sum_{i=1}^N{\delta_{a(N)_i/N}}$, observe that $r_N = N\E[\mu_N]$.
Since $\mu_N \to \mu$ weakly, in particular we have $r_N/N \to \E[\mu]$, as $N \to \infty$.

Define
$$b(N) := (a(N)_1 - r_N, \ldots, a(N)_N - r_N) \in \PP_N,\quad N \geq 1.$$
It is clear that $\{ b(N) \}_{N \geq 1}$ is a regular sequence of ordered tuples such that $\sum_{i=1}^N{b(N)_i} = 0$; moreover, one can verify that its limiting measure $\nu$ is the pushforward of $\mu$ under the map $x \mapsto x - \E[\mu]$ (this uses the limit $r_N/N \to \E[\mu]$).
The first two moments of $\nu$ are related to those of $\mu$ by
$$
\E[\nu] = 0, \quad \Var[\nu] = \Var[\mu].
$$
By the assumption, we have the limit
\begin{equation}\label{simplification2}
\lim_{N \rightarrow \infty}{\B_{b(N)}\left( \frac{y_1}{\sqrt{N}}, \ldots, \frac{y_m}{\sqrt{N}}; N; \theta \right)} =
\exp\left( \frac{\Var[\nu]}{2\theta} \sum_{i=1}^m{y_i^2} \right) =
\exp\left( \frac{\Var[\mu]}{2\theta} \sum_{i=1}^m{y_i^2} \right).
\end{equation}
From Corollary $\ref{cor:additionlabel}$, it follows that
\begin{align}
\B_{b(N)}\left( \frac{y_1}{\sqrt{N}}, \ldots, \frac{y_m}{\sqrt{N}}; N; \theta \right)
&= \B_{a(N)}\left( \frac{y_1}{\sqrt{N}}, \ldots, \frac{y_m}{\sqrt{N}}; N; \theta \right)
\exp\left( - \frac{r_N}{\sqrt{N}} \sum_{i=1}^m{y_i} \right)\nonumber\\
&= \B_{a(N)}\left( \frac{y_1}{\sqrt{N}}, \ldots, \frac{y_m}{\sqrt{N}}; N; \theta \right)
\exp\left( - \sqrt{N}\cdot\E[\mu_N] \sum_{i=1}^m{y_i} \right)\label{simplification3}
\end{align}
The general conclusion of Theorem $\ref{thm:asymptoticbessels}$ then follows from $(\ref{simplification2})$ and $(\ref{simplification3})$.
\end{proof}

\subsection{Proof of Theorem \ref{thm:asymptoticbessels} for $m = 1$}\label{sec:m1}

Let $\{ a(N)\in\PP_N \}_{N \geq 1}$ be a regular sequence of ordered tuples with limiting measure $\mu$ that satisfies the conditions in the statement of Theorem $\ref{thm:asymptoticbessels}$.
Because of Lemma $\ref{lem:meanzero}$, we do not lose any generality by assuming $\E[\mu_N] = \sum_{i=1}^N{a(N)_i}/N^2 = 0$, for all $N\in\N_+$, so let us assume this. Since $\mu_N \to \mu$ weakly, then also $\E[\mu] = 0$.

In this subsection, we finish the proof of Theorem \ref{thm:asymptoticbessels} by showing
\begin{equation}\label{limit1}
\lim_{N\rightarrow\infty}
\B_{a(N)}\left( y/\sqrt{N}; N; \theta \right)
= \exp\left( \frac{y^2 \Var[\mu]}{2\theta} \right)
\end{equation}
uniformly for $y$ belonging to compact subsets of $\C$.
We proceed in several steps.

\medskip

\textit{Step 1.} 
Let us assume that $(\ref{limit1})$ holds uniformly for $y$ belonging to compact subsets of $\R^* = \R\setminus\{0\}$, and moreover that the sequence
\begin{equation}\label{bddsequence}
\left\{\B_{a(N)}\left( y/\sqrt{N}; N; \theta \right) \right\}_{N \geq 1}
\end{equation}
is uniformly bounded in a real neighborhood of zero.
Under these assumptions, we claim that $(\ref{limit1})$ holds uniformly for $y$ belonging to compact subsets of $\C$.

From Corollary $\ref{cor:bounded2}$, we have
\begin{equation}\label{ineq1}
\left|  \B_{a(N)}\left( y/\sqrt{N}; N; \theta \right) \right|
\leq \B_{a(N)}\left( \Re y/\sqrt{N}; N; \theta \right).
\end{equation}
Our assumptions imply that the right hand side of $(\ref{ineq1})$ is uniformly bounded for $y$ on compact subsets of $\C$, and therefore so is the sequence $(\ref{bddsequence})$.
As a consequence of Montel's theorem, any subsequence of $(\ref{bddsequence})$ has a sub-subsequence which converges uniformly (on compact subsets of $\C$) to an entire function.
But each such entire function must coincide with $\exp(y^2/(2\theta \Var[\mu]))$ on $\R^*$ because of the assumption.
Since $\exp(y^2/(2\theta \Var[\mu]))$ is an entire function, analytic continuation shows that all subsequential limits of $(\ref{bddsequence})$ coincide and are equal to $\exp(y^2/(2\theta \Var[\mu]))$.
Our claim is proved.

\smallskip

\textit{Thus the problem is reduced to proving $(\ref{limit1})$ uniformly for $y$ on compact subsets of $\R^*$ and the uniform boundedness of $(\ref{bddsequence})$ in a real neighborhood of the origin.} They are proved in the next steps: Steps 2--7 show the uniform limit $(\ref{limit1})$ for $y$ on compact subsets of $\R^*$, while Step 8 shows the uniform boundedness of $(\ref{bddsequence})$ near the origin.

\smallskip

Before moving on, let us fix a large enough interval $J = [-s, s] \subset \R$, $s > 0$, such that
\begin{equation*}
\text{supp}{(\mu)} \subseteq J\quad \text{ and }\quad\{ a(N)_i/N : 1\leq i \leq N, \ N\in\N_+ \} \subseteq J.
\end{equation*}

\smallskip

\textit{Step 2.}
With the integral representation in Theorem $\ref{besselthm2}$ and the method of steepest descent, we will prove the desired limit $(\ref{limit1})$ uniformly for $y$ belonging to compact subsets of $(0, \infty)$.
A similar analysis, by using instead Theorem $\ref{besselthm1}$, shows that $(\ref{limit1})$ also holds uniformly for $y$ belonging to compact subsets of $(-\infty, 0)$.
We only prove $(\ref{limit1})$ uniformly for $y$ belonging to compact subsets of $(0, \infty)$, as the case of $(-\infty, 0)$ is very similar.

Let $y > 0$ be a real variable and $N$ be large enough so that $N\theta > 1$.
In this situation, the contour in Theorem $\ref{besselthm2}$ can be deformed to a vertical line to the right of the poles (and directed upwards), by virtue of Cauchy's theorem.
The key point is that the exponential $\exp(yz)$ has a bounded modulus, so the integrand in $(\ref{besselthm2eqn})$ grows as $z^{-N\theta}$, when $|z|$ is large.
After the change of variables $y \mapsto y/\sqrt{N}$ and $z \mapsto Nz$, we can write
\begin{equation}\label{eqn:besselproof1}
\B_{a(N)}\left(\frac{y}{\sqrt{N}}; N; \theta\right)
= \frac{\Gamma(\theta N) \sqrt{N}^{1-\theta N}}{y^{\theta N - 1}\cdot 2\pi\ii}
\int^{x_0 + \ii \infty}_{x_0 - \ii \infty}
\exp\left( N H_N(z; \mu, \theta) \right)
\phi_N(z; a(N); \theta)dz,
\end{equation}
where
\begin{equation}\label{Hwphi}
\begin{gathered}
H_N(z; \mu, \theta) := \frac{yz}{\sqrt{N}} - w(z),\quad\quad w(z) := \theta\int_{\R}{\ln{(z - t)}\mu(dt)},\\
\phi_N(z; a(N); \theta) := e^{Nw(z)}
\prod_{i=1}^N{\left( z - \frac{a(N)_i}{N} \right)^{- \theta}}.
\end{gathered}
\end{equation}
The point $x_0\in\R$ is arbitrary, as long as it is far enough to the right; it suffices to have $x_0 > s$.
The logarithms (in $w(z)$ and the one needed for $x^{-\theta}$ in $\phi(z; a(N); \theta)$) are defined on $\C\setminus (-\infty, 0]$. This choice poses no problem because $\inf_{t\in\text{supp}(\mu)}{\Re(z - t)} > \Re z - \sup_{x\in \text{supp}(\mu)}{x} \geq \Re z - \sup_{x\in J}{x} = x_0 - \sup_{x\in J}{x} > 0$, for all $z$ in our contour.
Similarly, $\Re(z - a(N)_i/N) > 0$, for $i = 1, \ldots, N$, $N \in \N_+$.

In all that follows, we will assume $N > 1/\theta$ so that we can use $(\ref{eqn:besselproof1})$.

\smallskip

\textit{Step 3.}
Here we prove
\begin{equation}\label{boundPhi}
|\ln{\phi_N(z; a(N), \theta)}| = o(1), \quad N \to \infty;
\end{equation}
the bound is uniform over $z\in\C$ such that $\Re z \geq c \sqrt{N}$, for any fixed constant $c>0$.

Recall $\mu_N = \frac{1}{N}\sum_{i=1}^N{\delta_{a(N)_i/N}}$.
From $(\ref{Hwphi})$, we need to bound the absolute value of
\begin{align*}
\ln{\phi_N(z; a(N), \theta)} = Nw(z) - \theta \sum_{i=1}^N{\ln\left( z - \frac{a(N)_i}{N} \right)}
&= N\theta \left\{ \int_{\R}{\ln{(z - t)}\mu(dt)} - \int_{\R}{\ln{(z - t)}\mu_N(dt)} \right\}\\
&= N\theta \left\{ \int_{\R}{\ln{\left( 1 - \frac{t}{z} \right)}\mu(dt)} - \int_{\R}{\ln{\left( 1 - \frac{t}{z} \right)}\mu_N(dt)} \right\}\\
&= N\theta \left\{ \frac{m_2(\mu_N) - m_2(\mu)}{2z^2} + \sum_{k=3}^{\infty}{\frac{m_k(\mu_N) - m_k(\mu)}{kz^k}} \right\}.
\end{align*}
Above we denoted the $k$-th moment of a measure $\nu$ by $m_k(\nu)$.
We also used our assumption that $m_1(\mu_N) = m_1(\mu) = 0$.

By condition, all measures $\mu, \mu_N$ ($N\in\N_+$) are supported on $J = [-s, s]$, and $\mu_N \rightarrow \mu$ weakly as $N \rightarrow \infty$.
The former fact implies $|m_k(\mu_N) - m_k(\mu)| \leq 2s^k$, for any $k \geq 1$.
Therefore, for $\Re z > c\sqrt{N}$ (so that also $|z| > c\sqrt{N}$) and $N$ large enough so that $c\sqrt{N} > 2s$, we have
\begin{multline}\label{diff_needbound}
\left| \ln{\phi_N(z; a(N), \theta)} \right| \leq
N\theta \left\{ \frac{|m_2(\mu_N) - m_2(\mu)|}{2|z|^2} + \sum_{k=3}^{\infty}{\frac{2s^k}{k|z|^k}} \right\}\\
\leq N\theta \left\{ \frac{|m_2(\mu_N) - m_2(\mu)|}{2|z|^2} + \sum_{k=3}^{\infty}{\left( \frac{s}{|z|} \right)^k} \right\}
\leq N\theta \left\{ \frac{|m_2(\mu_N) - m_2(\mu)|}{2|z|^2} + \frac{s^3}{|z|^2(|z| - s)} \right\}\\
\leq N\theta \left\{ \frac{|m_2(\mu_N) - m_2(\mu)|}{2c^2N} + \frac{2s^3}{c^3 N^{3/2}} \right\}
= \frac{\theta |m_2(\mu_N) - m_2(\mu)|}{2c^2} + \frac{2\theta s^3}{c^3\sqrt{N}}.
\end{multline}
Since $\mu_N \to \mu$ weakly, then $|m_2(\mu_N) - m_2(\mu)| \to 0$ as $N \rightarrow \infty$.
Then \eqref{boundPhi} follows from \eqref{diff_needbound}.

\smallskip

\textit{Step 4.}
We next study the asymptotics of the integral in $(\ref{eqn:besselproof1})$ for a specific choice of $x_0$: we let $x_0$ be the critical point of $H_N(z; \mu, \theta)$ that we will denote by $z_0$ henceforth. The critical point is the solution to $H_N'(z; \mu, \theta) = 0$, i.e., it satisfies:
\begin{equation}\label{eqn:critical}
\frac{y}{\theta\sqrt{N}} = \int_{\R}{\frac{\mu(dt)}{z_0 - t}}.
\end{equation}
We claim that there is a unique real solution $z_0$ such that $z_0 > \sup_{x\in\text{supp}(\mu)}{x}$ for the equation $(\ref{eqn:critical})$.
When $z_0$ is very close to $a_+ := \sup_{x\in\text{supp}(\mu)}{x}$, but larger than it, the right hand side of \eqref{eqn:critical} is close to $+\infty$.
Also, the right hand side of \eqref{eqn:critical} monotonically decreases as $z_0$ increases, and it tends to zero when $z_0$ tends to $+\infty$.
This argument proves the claim that there is a unique real solution $z_0$ to $(\ref{eqn:critical})$ such that $z_0 > a_+$.
After expanding the right hand side of $(\ref{eqn:critical})$, we have:
\begin{equation}\label{expansion1}
\frac{y}{\theta\sqrt{N}} = \int_{\R} \frac{\mu(dt)}{z_0 - t} = \frac{1}{z_0} + \frac{m_2}{z_0^3} + \frac{m_3}{z_0^4} + \dots,
\end{equation}
where
$$m_k := \int_{\R}{t^k \mu(dt)}, \ k \geq 1;$$
in particular, we have
\begin{equation}\label{estimate1}
z_0 = \frac{\theta\sqrt{N}}{y} (1 + o(1)).
\end{equation}
The bound $o(1)$ is uniform over $y$ belonging to compact subsets of $(0, \infty)$.
Note that we used $m_1 = \E[\mu] = 0$.

\smallskip

\textit{Step 5.}
We estimate the integral in $(\ref{eqn:besselproof1})$ by reducing the contour to an $\epsilon$-neighborhood of $z_0$:
\begin{equation}\label{approx:1}
(\ref{eqn:besselproof1}) \approx \int^{z_0 + \ii \epsilon}_{z_0 - \ii \epsilon}
\exp\left( N H_N(z; \mu, \theta) \right)
\phi_N(z; a(N); \theta)dz.
\end{equation}
It is convenient to choose $\epsilon = N^{1/7}$, so $\epsilon$ grows to infinity with $N$.
Since $z_0$ is the same order as $N^{1/2}$, see $(\ref{estimate1})$, then the quantities $\ln{(z - t)}$, $t\in\text{supp}(\mu)$, and $\ln{(z - a(N)_i/N)}$, $i = 1, 2, \ldots, N$,  are defined for all $z$ in the complex $\epsilon$-neighborhood of $z_0$.

We want to prove that the relative error of this approximation is $o(1)$.
This is a consequence of the following facts that we will prove:
\begin{align*}
\left| \left( \int_{z_0 + \ii\epsilon}^{z_0 + \ii\infty} + \int_{z_0 - \ii\infty}^{z_0 - \ii\epsilon} \right){\exp(NH_N(z)) \phi_N(z) dz} \right|
&\leq \frac{C_1}{N} |\exp(NH_N(z_0))|,\\
\left|\int_{z_0 - \ii \epsilon}^{z_0 + \ii \epsilon}{\exp(NH_N(z))\phi_N(z) dz}\right| &\geq C_2|\exp(NH_N(z_0))|;
\end{align*}
above, $C_1, C_2 > 0$ are constants independent of $N$, and we denoted $H_N(z; \mu, \theta)$ and $\phi_N(z; a(N); \theta)$ simply by $H_N(z)$ and $\phi(z)$, respectively.
Observe that $\phi_N(z) = 1 + o(1)$ for all $z$ in the contour $\{\Re z = \Re z_0\}$ --- this is because of $(\ref{estimate1})$ and $(\ref{boundPhi})$.
As a result, one can check that the two previous inequalities would follow if we proved the estimate
\begin{align}
\left( \int_{z_0 + \ii\epsilon}^{z_0 + \ii\infty} + \int_{z_0 - \ii\infty}^{z_0 - \ii\epsilon} \right){ \left| \exp\{N (H_N(z) - H_N(z_0))\} \right| dz}
&\leq \frac{C_1}{N},\label{step3.2}
\end{align}
and also showed the existence of two constants $C_2', C_2'' > 0$ such that
\begin{align}
\left|\int_{z_0 - \ii \epsilon}^{z_0 + \ii \epsilon}{\exp\{N (H_N(z) - H_N(z_0))\} dz}\right| &\geq C_2',\label{step3.1}\\
\int_{z_0 - \ii \epsilon}^{z_0 + \ii \epsilon}{\left| \exp\{N (H_N(z) - H_N(z_0))\} \right| dz} &\leq C_2''.\label{step3.31}
\end{align}

First we prove $(\ref{step3.2})$.
From $(\ref{estimate1})$, $z_0$ is of the same order of magnitude as $\sqrt{N}$, so $2z_0^2 > (z_0 - t)^2$, for all $t\in\text{supp}(\mu)$, as long as $N$ is large enough.
Then we obtain for any $h\in\R$:
\begin{equation}\label{boundreal}
\Re H_N(z_0 + \ii h) - \Re H_N(z_0) = -\frac{\theta}{2}\int_{\R}{ \ln\left( 1 + \frac{h^2}{(z_0 - t)^2} \right) \mu(dt) }
\leq -\frac{\theta}{2} \ln \left( 1 + \frac{h^2}{2z_0^2} \right).
\end{equation}
It follows that
\begin{equation}\label{step3.3}
\textrm{LHS of }(\ref{step3.2}) \leq \int_{|h| \geq \epsilon} { \exp\left\{ -\frac{\theta N}{2} \ln \left( 1 + \frac{h^2}{2z_0^2} \right) \right\}  dh} = 2 \int_{\epsilon}^{\infty}{ \exp\left\{ -\frac{\theta N}{2} \ln \left( 1 + \frac{h^2}{2z_0^2} \right) \right\} dh}.
\end{equation}
We will break the last integral into two.
It is easy to verify that $\ln(1+x) \geq x/2$, for all $x\in[0, 2]$.
Then one of the integrals will be from $h = \epsilon$ to $h = 2z_0$; we can estimate:
\begin{equation}\label{step3.4}
\int_{\epsilon}^{2z_0}{ \exp\left\{ -\frac{\theta N}{2} \ln \left( 1 + \frac{h^2}{2z_0^2} \right) \right\} dh}
\leq \int_{\epsilon}^{2z_0}{ \exp\left\{ -\frac{\theta N h^2}{8z_0^2} \right\} dh}
\leq 2z_0 \exp\left\{ -\frac{\theta N \epsilon^2}{8z_0^2} \right\}
\end{equation}
and this is exponentially small as a function of $N$ because $z_0 = \frac{\theta\sqrt{N}}{y}(1 + o(1))$ and $\epsilon = N^{1/7}$.
The second integral will be from $h = 2z_0$ to $h = \infty$; we simply use $\ln(1 + x) \geq \ln{x}$, for $x \geq 0$, to get
\begin{equation}\label{step3.5}
\int_{2z_0}^{\infty}{ \exp\left\{ -\frac{\theta N}{2} \ln \left( 1 + \frac{h^2}{2z_0^2} \right) \right\} dh}
\leq \int_{2z_0}^{\infty}{ \left( \frac{h^2}{2z_0^2} \right)^{-\theta N/2} dh}
= \sqrt{2}z_0 \int_{\sqrt{2}}^{\infty}{ h^{-\theta N} dh}
= \frac{z_0 \sqrt{2}^{2-\theta N}}{\theta N - 1}
\end{equation}
and the latter is also exponentially small as a function of $N$.
Putting together $(\ref{step3.3})$, $(\ref{step3.4})$ and $(\ref{step3.5})$, we obtain $(\ref{step3.2})$; the constant $C_1 > 0$ can be any positive real number.

Next we show the existence of $C_2'' > 0$ satisfying \eqref{step3.31}.
From \eqref{boundreal}, the bound $z_0 = \frac{\theta\sqrt{N}}{y}(1 + o(1)) \leq \frac{2\theta\sqrt{N}}{y}$ for large $N$, and the same estimate used to obtain \eqref{step3.4} (namely $\ln(1+x) \geq x/2$), we deduce
$$
\text{LHS of \eqref{step3.31}} \leq \int_{-\epsilon}^{\epsilon}{\exp\left( -\frac{\theta N h^2}{8z_0^2} \right) dh}
\leq \int_{-\epsilon}^{\epsilon}{\exp\left( -\frac{y^2}{32\theta}\cdot h^2 \right)dh} \leq \int_{-\infty}^{\infty}{\exp\left( -\frac{y^2}{32\theta}\cdot h^2 \right)dh}.
$$
Hence we can set $C_2'' := \int_{-\infty}^{\infty}{e^{-kh^2}dh}$ with $k := y^2/(32\theta)$.

Finally we verify $(\ref{step3.1})$.
From Taylor's expansion, we have for any $h\in\R$ with $|h| \leq \epsilon$:
$$
H_N(z_0 + \ii h) - H_N(z_0) = \frac{h^2}{2}w''(z_0) + h^3\delta;\quad
|\delta| \leq \frac{1}{6}\sup_{z\in [z_0 - \ii \epsilon, z_0 + \ii \epsilon]}|w'''(z)|.
$$
The exact form of the function $w(z)$, and $(\ref{estimate1})$, give
$$
w''(z_0) = -\theta \int_{\R}{ \frac{\mu(dt)}{(z_0 - t)^2} } = -\frac{y^2}{\theta N}(1 + o(1)),\
\sup_{z\in[z_0-\ii\epsilon, z_0+\ii\epsilon]}{|w'''(z)|} \leq 2\theta\int_{\R}{\frac{\mu(dt)}{(z_0 - t)^3}}
\leq CN^{-3/2},
$$
for some constant $C>0$.
Therefore, whenever $|h| \leq \epsilon = N^{1/7}$, we have $|Nh^3\delta| \leq CN^{-1/14}$ and $e^{Nh^3\delta} = 1 + o(1)$.
As a result,
$$
\textrm{LHS of }(\ref{step3.1}) =
\left| \int_{-\epsilon}^{\epsilon}{\exp\left\{ -\frac{y^2h^2}{2\theta}(1 + o(1)) \right\} (1+o(1)) dh} \right|
\geq \frac{1}{2}\int_{-\epsilon}^{\epsilon}{ \exp\left\{ -\frac{y^2 h^2}{\theta} \right\} dh}
\geq \frac{1}{2}\int_{-1}^{1}{ e^{-y^2 h^2/\theta} dh}
$$
and this proves $(\ref{step3.1})$, as we can set $C_2' := \frac{1}{2}\int_{-1}^{1}{ e^{-y^2 h^2/\theta} dh}$.

\smallskip

\textit{Step 6.}
We proceed to estimate the formula $(\ref{approx:1})$ from the last step (recall $\epsilon = N^{1/7}$).
One can easily check that $w''(z_0) < 0$, so we can make the following change of variables:
$$u = \ii\sqrt{-w''(z_0)};\ \ z = z_0 - \frac{t}{u\sqrt{N}}, \ t\in [-\sqrt{N}|u|\epsilon, \sqrt{N}|u|\epsilon].$$
We also need the Taylor expansion in the complex $\epsilon$-neighborhood of $z_0$:
$$
H_N(z_0 + h) = H_N(z_0) - \frac{h^2}{2}w''(z_0) + h^3\delta;\ \ \
|\delta| \leq \frac{1}{6}\sup_{z\in [z_0, z_0 + h]}|w'''(z)|.
$$
Then the integral $(\ref{approx:1})$ can be estimated as follows:
\begin{align}
(\ref{approx:1}) &= \frac{\exp\left( N H_N(z_0) \right)}{-u\sqrt{N}}\int^{\sqrt{N}|u|\epsilon}_{-\sqrt{N}|u|\epsilon}
\exp\left( N \left( H_N\left(z_0 - \frac{t}{u\sqrt{N}}\right) - H_N(z_0) \right) \right)
\phi_N \left( z_0 - \frac{t}{u\sqrt{N}} \right)dt \nonumber\\
&= \frac{\exp\left( N H_N(z_0) \right)}{-u\sqrt{N}}\int^{\sqrt{N}|u|\epsilon}_{-\sqrt{N}|u|\epsilon}
\exp\left( -\frac{t^2}{2} + \frac{t^3 \widetilde{\delta}}{N^{1/2}} \right)
\phi_N \left( z_0 - \frac{t}{u\sqrt{N}} \right) dt \nonumber\\
&\approx \frac{\exp\left( N H_N(z_0) \right) }{-u\sqrt{N}} \int^{\sqrt{N}|u|\epsilon}_{-\sqrt{N}|u|\epsilon}
\exp\left( -\frac{t^2}{2} + \frac{t^3 \widetilde{\delta}}{N^{1/2}} \right) dt \label{est1}\\
&\approx \frac{\exp\left( N H_N(z_0) \right) }{-u\sqrt{N}} \int^{\sqrt{N}|u|\epsilon}_{-\sqrt{N}|u|\epsilon}
\exp\left( -\frac{t^2}{2} \right) dt \label{est2}\\
&\approx \frac{\exp\left( N H_N(z_0) \right) }{-u\sqrt{N}}
\int^{+\infty}_{-\infty} \exp\left( -\frac{t^2}{2} \right) dt
= \frac{\exp\left( N H_N(z_0) \right) \sqrt{2\pi}}{-u\sqrt{N}},\label{est3}
\end{align}
where above we denoted
$$\widetilde{\delta} := -\delta/u^3.$$

Plugging the last estimate into $(\ref{eqn:besselproof1})$, and by Stirling's approximation (which leaves a relative error of order $O(1/N)$), we obtain
\begin{equation}\label{approximation1}
\B_{a(N)} \left( \frac{y}{\sqrt{N}}; N; \theta \right) \approx
\frac{\theta^{\theta N - \frac{1}{2}} \sqrt{N}^{\theta N - 1} e^{-\theta N}\exp\left( N H_N(z_0) \right)}{y^{\theta N - 1}\sqrt{-w''(z_0)}}.
\end{equation}

In $(\ref{est1})$, we used Step 3 and $(\ref{estimate1})$, which show that $\phi_N(z_0 + h) = 1 + o(1)$, for all $|h| \leq \epsilon = N^{1/7}$.
The relative error in this approximation is $o(1)$.

For $(\ref{est2})$, using the estimate for $|\delta|$ in Step 5, we obtain for all $t\in\R$ with $|t| \leq \sqrt{N}|u|\epsilon$:
$$ \left| \frac{t^3\widetilde{\delta}}{N^{1/2}} \right| \leq \frac{N^{3/2}|u|^3\epsilon^3 |\widetilde{\delta}|}{N^{1/2}}
= N\epsilon^3 |\delta| \leq N\epsilon^3 (CN^{-3/2}) = CN^{-1/14},$$
where $C>0$ is a constant.
Therefore
$$\exp\left( \frac{t^3\widetilde{\delta}}{N^{1/2}} \right) = 1 + o(1),$$
which shows the relative error in the approximation $(\ref{est2})$ is $o(1)$.

Finally, Step 5 also shows that $|u| = |w''(z_0)|^{1/2}$ is of order $N^{-1/2}$.
Then $\sqrt{N}|u|\epsilon$ is of order $\epsilon = N^{1/7}$, and therefore it is clear that the estimate $(\ref{est3})$ also gives a relative error of order $o(1)$.

Consequently, the estimation $(\ref{approximation1})$ is off by a negligible relative error of $o(1)$.

\smallskip

\textit{Step 7.}
We still have to take the limit of the right hand side of $(\ref{approximation1})$, as $N \rightarrow \infty$, but this is a simple task.
We will use the approximation $\ln{(1+x)} = x - x^2/2 + O(x^3)$, when $x \rightarrow 0$, without further mention.
The definition of $w(z)$ and $(\ref{estimate1})$ give
\begin{equation}\label{estimate3}
-w''(z_0) = \frac{y^2}{\theta N} (1 + o(1)).
\end{equation}
Next we need to estimate $H_N(z_0)$; by replacing $(\ref{estimate1})$ into $(\ref{expansion1})$, we have
\begin{equation}\label{estimate2}
\frac{y}{\theta \sqrt{N}} = \frac{1}{z_0} +  \frac{m_2}{z_0^3} + O(N^{-2}).
\end{equation}
Therefore
\begin{equation}\label{estimate4}
\begin{gathered}
H_N(z_0) = \frac{yz_0}{\sqrt{N}} - w(z_0) = \left\{ \theta + \frac{\theta m_2}{z_0^2} + O(N^{-3/2}) \right\}
- \theta \ln{z_0} - \theta\int_{\R}{\ln\left(1 - \frac{t}{z_0} \right) \mu(dt)}\\
= \theta - \theta \ln{z_0} + \frac{3\theta m_2}{2z_0^2} + O(N^{-3/2}).
\end{gathered}
\end{equation}

Now we can replace $(\ref{estimate3})$, $(\ref{estimate2})$ and $(\ref{estimate4})$ into $(\ref{approximation1})$, which gives
\begin{multline}\label{estimate6}
\B_{a(N)} \left( \frac{y}{\sqrt{N}}; N; \theta \right) = \sqrt{N}^{\theta N} (\theta e^{-1} y^{-1})^{\theta N}\exp\left\{ N \left(  \theta - \theta \ln{z_0} + \frac{3\theta m_2}{2z_0^2}  \right) \right\}\times (1 + o(1))\\
= \sqrt{N}^{\theta N}\exp\left\{ N \left(  \theta\ln{\theta} - \theta \ln{(yz_0)} + \frac{3\theta m_2}{2z_0^2}  \right) \right\}\times (1 + o(1)).
\end{multline}
From $(\ref{estimate2})$, we have
$$yz_0 = \theta\sqrt{N}\left( 1 + \frac{m_2}{z_0^2} \right) \times (1 + O(N^{-3/2})),$$
then
$$\ln(yz_0) = \ln{\theta} + \frac{1}{2}\ln{N} + \frac{2m_2}{2z_0^2} + O(N^{-3/2}).$$
Back into $(\ref{estimate6})$, it follows that
\begin{equation}\label{estimate5}
\B_{a(N)}\left( \frac{y}{\sqrt{N}}; N; \theta \right) = \sqrt{N}^{\theta N} \exp\left\{ \theta N \left( -\frac{\ln{N}}{2} + \frac{m_2}{2z_0^2} \right) \right\} \times (1 + o(1)).
\end{equation}
Finally because of $(\ref{estimate2})$ and $(\ref{estimate1})$, we obtain
\begin{gather*}
\frac{m_2}{2z_0^2} = \frac{m_2}{2} \left( \frac{y^2}{\theta^2 N} \right) + O(N^{-3/2})
= \frac{m_2 y^2}{2\theta^2 N} + O(N^{-3/2})
\end{gather*}
and back into $(\ref{estimate5})$ gives
\begin{equation*}
\B_{a(N)}\left( \frac{y}{\sqrt{N}}; N; \theta \right) = \sqrt{N}^{\theta N} \exp\left\{ -\frac{\theta N\ln{N}}{2} + \frac{m_2 y^2}{2\theta} \right\}(1 + o(1)) = \exp\left\{ \frac{m_2 y^2}{2\theta} \right\}(1+o(1)).
\end{equation*}
Since $m_2 = \int_{\R}{t^2\mu(dt)} = \Var[\mu]$, the uniform limit $(\ref{limit1})$ is now proved for $y$ in compact subsets of $(0, \infty)$, concluding one of the tasks laid out in Step 1.

\smallskip

\textit{Step 8.}
It remains to complete the second task laid out in Step 1, which is to prove the uniform boundedness of $(\ref{bddsequence})$ in a real neighborhood of the origin.
Given $x, R\in\R$ with $|x| \leq R$ and $b\in\R$, we have the obvious inequalities
\begin{equation}\label{eqn:step8.1}
e^{xb} \leq e^{|xb|} \leq e^{R|b|} \leq e^{Rb} + e^{-Rb}.
\end{equation}
On the other hand, Proposition $\ref{prop:combinatorial}$ gives
\begin{equation}\label{eqn:step8.2}
\B_{a(N)}(x; N; \theta) = \EE_{\{a_i^{(k)}\}} \left[ e^{x a^{(1)}_1} \right],
\end{equation}
where the expectation is over the orbital beta process $\{ a^{(k)}_i \}_{1 \leq i \leq k \leq N-1}$ with top level $a(N)$; in particular, this implies that $\B_a(x; N; \theta)$ is a positive real when $x\in\R$.
Apply $(\ref{eqn:step8.1})$ to $b = a^{(1)}_1$; then applying $(\ref{eqn:step8.2})$ several times yields the bound
\begin{multline*}
\sup_{|x| \leq R}|{\B_{a(N)}(x; N; \theta)|} = \sup_{|x| \leq R}{\B_{a(N)}(x; N; \theta)} = \sup_{|x| \leq R}{\EE_{\{a_i^{(k)}\}} \left[ e^{x a^{(1)}_1} \right]}\\
\leq \EE_{\{a_i^{(k)}\}} \left[ e^{R a^{(1)}_1} \right] + \EE_{\{a_i^{(k)}\}} \left[ e^{-R a^{(1)}_1} \right]
= \B_{a(N)}(R; N; \theta) + \B_{a(N)}(-R; N; \theta).
\end{multline*}
But both sequences $\{ \B_{a(N)}(R; N; \theta) \}_{N\geq 1}$, $\{ \B_{a(N)}(-R; N; \theta) \}_{N\geq 1}$ are bounded because they both converge to $\exp(\Var[\mu]R^2/2\theta)$.
It follows that the sequence $(\ref{bddsequence})$ is uniformly bounded on the compact set $[-R, R]$, thus completing the second task left in Step 1.
We have finished the proof of Theorem $\ref{thm:asymptoticbessels}$ for $m = 1$. \hfill \qed

\subsection{Proof of Theorem $\ref{thm:asymptoticbessels}$ for general $m$}

\begin{lem}\label{lem:boundedness}
Let $K \subseteq \R^m$ be any compact set, and $\{a(N) \in \PP_N\}_{N \geq 1}$ be a regular sequence of ordered tuples and such that $\sum_{i=1}^N{a(N)_i} = 0$ for all $N\in\N_+$, then
\begin{equation*}
\sup_{N \geq m}\sup_{(y_1, \dots, y_m)\in K}{\left| \B_{a(N)}\left( \frac{y_1}{\sqrt{N}}, \dots, \frac{y_m}{\sqrt{N}}; N; \theta \right)  \right|} < \infty.
\end{equation*}
\end{lem}

\begin{proof}
We can (and do) assume $K = B[0, R] := \{ (x_1, \ldots, x_m)\in\R^m : x_1^2 + \ldots + x_m^2 \leq R^2 \}$, for some $R>0$.

\smallskip

\textit{Step 1.}
Let $a_1, \ldots, a_m, b_1, \ldots, b_m$ be any $2m$ real numbers.
We will need the following bound:
\begin{align*}
\exp\left( \sum_{i = 1}^m  {a_i b_i} \right) &\leq \exp\left( \sum_{i = 1}^m  {\max(a_i, -a_i) \max(|b_1|, \ldots, |b_m|)} \right)\\
&\leq \sum_{\epsilon_i \in \{\pm 1\}} \exp\left( \sum_{i = 1}^m  {\epsilon_i a_i \max(|b_1|, \ldots, |b_m|)} \right)\\
&\leq \sum_{\epsilon_i \in \{\pm 1\}} \sum_{r = 1}^m \exp\left( \sum_{i = 1}^m  {\epsilon_i a_i |b_r|} \right)= \sum_{\epsilon_i \in \{\pm 1\}} \sum_{r = 1}^m \exp\left( \sum_{i = 1}^m {\epsilon_i a_i b_r} \right).
\end{align*}

\textit{Step 2.}
Let $z_1, \ldots, z_m\in\C$ be arbitrary.
From Corollary $\ref{cor:bounded2}$ and Proposition $\ref{prop:combinatorial}$, we obtain
$$
|\B_{a(N)}(z_1, \ldots, z_m; N; \theta)| \leq
\EE_{\{a_j^{(i)}\}}\left[ \exp\left( \sum_{i = 1}^m \Re z_i \left( \sum_{j  = 1}^i {a_j^{(i)}} - \sum_{j = 1}^{i - 1} {a_j^{(i-1)}} \right) \right) \right],
$$
where the expectation is taken over the orbital beta corners process with top level $a(N)$.
From the inequality in step 1, applied to $a_i := \Re z_i$ and $b_i := \sum_{j = 1}^i {a_j^{(i)}} - \sum_{j = 1}^{i - 1}{a_j^{(i-1)}}$, we deduce
\begin{align*}
|\B_{a(N)}(z_1, \ldots, z_m; N; \theta)| &\leq \sum_{\epsilon_i\in\{\pm 1\}} \sum_{r = 1}^m \EE_{\{a_j^{(i)}\}}
\left[ \exp\left( \sum_{i = 1}^m \epsilon_i \Re z_i  \left( \sum_{j  = 1}^r {a_j^{(r)}} - \sum_{j = 1}^{r - 1} {a_j^{(r-1)}} \right) \right) \right]\\
&= \sum_{\epsilon_i\in\{\pm 1\}} \sum_{r = 1}^m \B_{a(N)}(\underbrace{0, \ldots, 0}_{r-1\textrm{ times}}, \epsilon_1 \Re z_1 + \ldots + \epsilon_m \Re z_m; N; \theta);
\end{align*}
the latter equality comes from Proposition $\ref{prop:combinatorial}$.

But since the Bessel function $\B_{a(N)}(; \theta)$ is symmetric on its arguments, we have
$$\B_{a(N)}(\underbrace{0, \ldots, 0}_{r-1\textrm{ times}}, \epsilon_1 \Re z_1 + \ldots + \epsilon_m \Re z_m; N; \theta) = \B_{a(N)}(\epsilon_1 \Re z_1 + \ldots + \epsilon_m \Re z_m; N; \theta).$$
The conclusion is that
$$|\B_{a(N)}(z_1, \ldots, z_m; N; \theta)| \leq m 2^m \sup_{\epsilon_i \in \{\pm 1\}}\B_{a(N)}(\epsilon_1 \Re z_1 + \ldots + \epsilon_m \Re z_m; N; \theta).$$

\textit{Step 3.}
Use the inequality just obtained for $z_i := y_i/\sqrt{N}$, $i = 1, \ldots, m$:
$$
\left| \B_{a(N)}\left( \frac{y_1}{\sqrt{N}}, \dots, \frac{y_m}{\sqrt{N}}; N; \theta \right) \right|
\leq m 2^m \sup_{\epsilon_i \in \{\pm 1\}}
\B_{a(N)}\left( \frac{\epsilon_1\Re y_1 + \ldots + \epsilon_m \Re y_m}{\sqrt{N}}; N; \theta \right).$$
Therefore, using also Theorem $\ref{thm:asymptoticbessels}$ for $m = 1$ (which was proved in the previous subsection):
$$
\sup_{N \geq m}\sup_{(y_1, \ldots, y_m)\in B[0, R]}\left| \B_{a(N)}\left( \frac{y_1}{\sqrt{N}}, \dots, \frac{y_m}{\sqrt{N}}; N; \theta \right) \right|
\leq m 2^m \sup_{N \geq m}\sup_{y\in B[0, mR]} \B_{a(N)}\left( \frac{y}{\sqrt{N}}; N; \theta \right) < \infty.$$
\end{proof}

\begin{cor}\label{cor:boundedness}
Let $K \subseteq \C^m$ be any compact set, and $\{a(N) \in \PP_N\}_{N \geq 1}$ be a regular sequence of ordered tuples and such that $\sum_{i=1}^N{a(N)_i} = 0$, for all $N\in\N_+$, then
\begin{equation*}
\sup_{N \geq m}\sup_{(y_1, \dots, y_m)\in K}{\left| \B_{a(N)}\left( \frac{y_1}{\sqrt{N}}, \dots, \frac{y_m}{\sqrt{N}}; N; \theta \right)  \right|} < \infty.
\end{equation*}
\end{cor}
\begin{proof}
The only difference between the statement of this corollary and that of Lemma \ref{lem:boundedness} is that now $K$ is a compact subset of $\C^{m+1}$ (and not of $\R^{m+1}$).
The result follows from Lemma \ref{lem:boundedness} and Corollary \ref{cor:bounded2}.
\end{proof}

\bigskip

First we prove the pointwise limit in the statement of Theorem $\ref{thm:asymptoticbessels}$, and later we obtain the general result by analytic continuation.

\begin{prop}\label{prop:masymptotics}
Let $y_1, y_2, \dots, y_m$ be real numbers such that
\begin{equation*}
\begin{gathered}
y_1 > y_2 > \dots > y_m > 0,\\
y_1 - y_2 > y_3,\ y_2 - y_3 > y_4, \dots,\ y_{m-2} - y_{m-1} > y_m.
\end{gathered}
\end{equation*}
Let $\{a(N)\in\PP_N\}_{N \geq 1}$ be a regular sequence of ordered tuples with limiting measure $\mu$ and such that $\sum_{i=1}^N{a(N)_i} = 0$, for all $N\in\N_+$.
Then
$$
\lim_{N\rightarrow\infty}{\B_{a(N)}\left( \frac{y_1}{\sqrt{N}}, \dots, \frac{y_m}{\sqrt{N}}; N; \theta \right)} =
\exp\left( \frac{\Var[\mu]}{2\theta}\sum_{i=1}^m{y_i^2} \right).
$$
\end{prop}
\begin{proof}
The proof of the proposition goes by induction on $m$; the base case $m = 1$ is proved in Section $\ref{sec:m1}$.
Assume the proposition holds up to some value $m \geq 1$, and let us prove it for $(m+1)$. Consider real numbers $y_1, \dots, y_m, y$ such that
\begin{equation}\label{proof:inequalities}
\begin{gathered}
y_1 > \dots > y_m > y > 0,\\
y_1 - y_2 > y_3, \ \dots, \ y_{m-2} - y_{m-1} > y_m, \ y_{m - 1} - y_m > y.
\end{gathered}
\end{equation}
The main tool will be Theorem $\ref{besselthm3}$.
We will use bold letters to denote vectors, e.g.,
$$\bfy := (y_1, \ldots, y_m, y),\quad \bfz := (z_1, \ldots, z_m).$$
We want to replace $y_k \mapsto y_k/\sqrt{N}$, $k = 1, \ldots, m$, $y \mapsto y/\sqrt{N}$, and make the change of variables $z_k \mapsto z_k/\sqrt{N}$, $k = 1, \ldots, m$, in Theorem $\ref{besselthm3}$.
From formulas $(\ref{defn:FGfns})$ for the functions $G_{\theta, \bfy}(\bfz)$ and $F_{\bfy}(\bfz)$, we have
$$G_{\theta, \bfy/\sqrt{N}}(\bfz/\sqrt{N}) = \sqrt{N}^{{m \choose 2} - \theta m^2} G_{\theta, \bfy}(\bfz),\quad F_{\bfy/\sqrt{N}}(\bfz/\sqrt{N}) = F_{\bfy}(\bfz).$$
(We are denoting $\bfy/\sqrt{N} := (y_1/\sqrt{N}, \ldots, y_m/\sqrt{N}, y/\sqrt{N})$ and $\bfz/\sqrt{N} := (z_1/\sqrt{N}, \ldots, z_m/\sqrt{N})$.)

Then we have
\begin{multline}\label{pieri1}
\B_{a(N)} \left( -\frac{y_1}{\sqrt{N}}, \ldots, -\frac{y_m}{\sqrt{N}}; N; \theta \right)
\B_{a(N)} \left( -\frac{y}{\sqrt{N}}; N; \theta \right) =
\frac{\Gamma(N\theta)}{\Gamma((N-m)\theta)\Gamma(\theta)^m }\\
\times\frac{\prod_{1 \leq i < j \leq m }{(y_i - y_j)^{1 - 2\theta}}}{y^{m\theta} (y_1\cdots y_m)^{\theta}}
\times\int\dots\int { G_{\theta, \bfy}(\bfz) F_{\bfy}(\bfz)^{\theta(N-m)-1} }\\
{\B_{a(N)}\left(- \frac{y_1 + z_1}{\sqrt{N}}, \ldots, -\frac{y_m + z_m}{\sqrt{N}}, -\frac{y-(z_1 + \dots + z_m)}{\sqrt{N}}; N; \theta\right)
\prod_{i=1}^m{(z_i^{\theta-1}dz_i)} },
\end{multline}
where $G_{\theta, \bfy}(\bfz)$, $F_{\bfy}(\bfz)$ are defined in $(\ref{defn:FGfns})$ and the domain of integration $\V_y \subset \R^m$ is the set of points $(z_1, \ldots, z_m)$ given by the inequalities
\begin{equation*}
\left\{
\begin{gathered}
z_1, \ldots, z_m \geq 0,\\
y \geq z_1 + \dots + z_m.
\end{gathered}
\right.
\end{equation*}
We want to take the limit of the right hand side of $(\ref{pieri1})$ as $N$ tends to infinity.
The factor before the integral is easy because of a well-known asymptotic result for the Gamma function; it gives
\begin{equation}\label{gamma1}
\frac{\Gamma(N\theta)}{\Gamma((N-m)\theta)} = (N\theta)^{m\theta}(1 + O(N^{-1})).
\end{equation}

Let us now focus on studying the asymptotics of the integral as $N$ goes to infinity.
Denote the integrand of \eqref{pieri1} by $H_{\theta, \bfy, a(N)}(\bfz)$.
Fix any $1/2 < \epsilon < 1$; it will be convenient to break up the integration domain of the integral of $(\ref{pieri1})$ into two parts:
$$[0, N^{-\epsilon}]^m \textrm{ and } \V_y\setminus [0, N^{-\epsilon}]^m.$$

\textit{Claim 1.} 
\begin{equation*}
\left| \int\dots\int_{\V_y \setminus [0, N^{-\epsilon}]^m}{ H_{\theta, \bfy, a(N)}(\bfz) d\bfz } \right| = o(N^{-m\theta}).
\end{equation*}

\textit{Claim 2.}
\begin{multline}\label{claim2}
\int_0^{N^{-\epsilon}}\dots \int_0^{N^{-\epsilon}}{ H_{\theta, \bfy, a(N)}(\bfz) d\bfz } =
(N\theta)^{-m\theta}\frac{\Gamma(\theta)^m y^{m\theta}\prod_{i=1}^m{y_i^{\theta}}}{\prod_{1 \leq i < j \leq m}(y_i - y_j)^{1 - 2\theta}}\\
\times(1 + O(N^{1 - 2\epsilon}))
\times \left( \B_{a(N)}\left( -\bfy/\sqrt{N}; N; \theta\right) + O(N^{-\epsilon}) \right).
\end{multline}
Let us complete the inductive step assuming the claims are true.
From the claims, $(\ref{pieri1})$ and $(\ref{gamma1})$:
\begin{multline}\label{eqn:forinduction}
\B_{a(N)}({-y_1/\sqrt{N}}, \dots, {-y_m/\sqrt{N}}; N; \theta)\B_{a(N)}({-y/\sqrt{N}}; N; \theta) = o(1)\\
+ (1 + O(N^{1 - 2\epsilon})) \times \left( \B_{a(N)}({-y_1/\sqrt{N}}, \dots, {-y_m/\sqrt{N}}, {-y/\sqrt{N}}; N; \theta) + O(N^{-\epsilon}) \right).
\end{multline}
From Lemma $\ref{lem:boundedness}$, we know
\begin{equation*}
\left| \B_{a(N)}\left( {-y_1/\sqrt{N}}, \dots, {-y_m/\sqrt{N}}, {-y/\sqrt{N}} ; N; \theta \right) \right| = O(1).
\end{equation*}
Therefore, recalling that $1/2 < \epsilon < 1$, $(\ref{eqn:forinduction})$ yields
\begin{multline}\label{eqn:asymptoticstatement}
\B_{a(N)}({-y_1/\sqrt{N}}, \dots, {-y_m/\sqrt{N}}; N; \theta) \B_{a(N)}({-y/\sqrt{N}}; N; \theta)\\
= \B_{a(N)}({-y_1/\sqrt{N}}, \dots, {-y_m/\sqrt{N}}, {-y/\sqrt{N}}; N; \theta) + o(1).
\end{multline}
With $(\ref{eqn:asymptoticstatement})$, our desired result can be proved for $(m+1)$, completing the inductive hypothesis and the proof.
It remains to prove the claims above.

\medskip

\textit{Proof of Claim 1.}
From the definition $(\ref{defn:FGfns})$ of $G_{\theta, \bfy}(\bfz)$ and the inequalities $(\ref{proof:inequalities})$ for $y_1, \ldots, y_m, y$, the function $G_{\theta, \bfy}(\bfz)$ is continuous on the compact set $\V_y$.
As a result,
$$\sup_{\bfz\in\V_y}{|G_{\theta, \bfy}(\bfz)|} < \infty.$$
From Lemma $\ref{lem:boundedness}$, one also has
\begin{equation*}
\sup_{N \geq m+1}\sup_{\bfz\in\V_y}{\left| \B_{a(N))}\left(-\frac{y_1 + z_1}{\sqrt{N}}, \dots, -\frac{y_m + z_m}{\sqrt{N}}, -\frac{y - (z_1 + \dots + z_m)}{\sqrt{N} }; N; \theta\right)\right|} < \infty.
\end{equation*}
From the estimates above, the claim is reduced to prove
\begin{equation}\label{eqn:claim1toprove}
\int\cdots\int{ | F_{\bfy}( \bfz ) |^{\theta(N - m) - 1} \prod_{i=1}^m{(z_i^{\theta - 1}dz_i)} } = o(N^{-m\theta}),
\end{equation}
where the domain of integration is $\V_y \setminus [0, N^{-\epsilon}]^m$.
The function $F_{\bfy}(\bfz)$ takes positive real values for $\bfz\in\V_y$, so we can drop the absolute value in $(\ref{eqn:claim1toprove})$. Next, for $i = 1, 2, \ldots, m$, one can verify
$$\frac{\partial}{\partial z_i}\ln(F_{\bfy}(\bfz)) = -\frac{1}{y - z_1 - \ldots - z_m} - \frac{1}{y_i + z_i} < 0.$$
Therefore, $F_{\bfy}(\bfz)$ is monotone decreasing with respect to its variables $z_1, \ldots, z_m$.
For any $\bfz\in\V_y\setminus [0, N^{-\epsilon}]^m$, there exists some $i \in \{1, \ldots, m\}$ such that $z_i > N^{-\epsilon}$.
In this case, let $\mathbf{e}_i$ be the vector in $\R^m$ with a $1$ in the $i$-th position and zeroes elsewhere; we then have
$$F_{\bfy}(\bfz) \leq F_{\bfy}(N^{-\epsilon}\mathbf{e}_i) = (1 - N^{-\epsilon}y^{-1})(1 + N^{-\epsilon}y_i^{-1}).$$
Then $(\ref{eqn:claim1toprove})$ would follow if we proved
$$\sum_{i=1}^m{ ((1 - N^{-\epsilon}y^{-1})(1 + N^{-\epsilon}y_i^{-1}))^{\theta(N-m)-1} \int \cdots \int {\prod_{j=1}^m{(z_j^{\theta - 1}dz_j)}} } = o(N^{-m\theta}),$$
where the domain of integration for the $i$-th summand is $\V_y \cap \{ \bfz \in \R_{\geq 0}^m \mid z_i \geq N^{-\epsilon} \}$.
But this bound is obvious in view of the fact that each of the $m$ integrals is bounded by
$$\int_0^y \cdots \int_0^y{\prod_{j = 1}^m{(z_j^{\theta - 1}dz_j)}} = (y^{\theta}/\theta)^m = O(1),$$
and that each equantity
$$\left((1 - N^{-\epsilon}y^{-1})(1 + N^{-\epsilon}y_i^{-1})\right)^{\theta(N-m)-1} =
\left( 1 - N^{-\epsilon}y^{-1}y_i^{-1}(y_i - y) + O(N^{-2\epsilon}) \right)^{\theta(N-m)-1}$$
is exponentially small in $N$.
Thus claim 1 is proved.

\smallskip

\textit{Proof of Claim 2.}
We estimate the integrand of $(\ref{claim2})$ when $0 \leq z_i \leq N^{-\epsilon}$, for $i = 1, \ldots, m$.
Denote
\begin{equation}\label{fNdef}
f_N(x_1, \ldots, x_{m+1}) := \B_{a(N)}\left( -\frac{x_1}{\sqrt{N}}, \ldots, -\frac{x_{m+1}}{\sqrt{N}} ; N, \theta \right).
\end{equation}
Corollary \ref{cor:boundedness} shows that $|f_N(x_1, \ldots, x_{m+1})|$ has a uniform upper bound for large $N$ and for $(x_1, \ldots, x_{m+1})$ varying over an arbitrary compact subset of $\C^{m+1}$.
As a result, the moduli of the partial derivatives $|(\partial_j f_N)(x_1, \ldots, x_{m+1})|$, $1 \leq j \leq m+1$, are also uniformly bounded for $N \geq m+1$ and for $(x_1, \ldots, x_{m+1})$ varying over an arbitrary compact subset of $\C^{m+1}$ (as a result of Cauchy's integral formula for derivatives).

From the mean value's theorem, one has
\begin{multline}\label{fNbound}
| f_N(y_1+z_1, \ldots, y_m+z_m, y-(z_1+\dots+z_m)) - f_N(y_1, \ldots, y_m, y) |\\
\leq \sum_{\substack{1 \leq i \leq m \\ 1 \leq j \leq m+1}}{|z_i| \sup_{t \in [0, 1]}{\left| \partial_j f_N (y_1 + tz_1, \ldots, y_m+tz_m, y - t(z_1 + \dots + z_m)) \right|} }.
\end{multline}
As mentioned above, the moduli $|(\partial_j f_N)(x_1, \ldots, x_{m+1})|$ are uniformly bounded for $(x_1, \ldots, x_{m+1})$ varying over the compact set $\{ (y_1 + tz_1, \ldots, y_m + tz_m, y - t(z_1 + \dots + z_m)) : t, z_1, \ldots, z_m\in [0, 1] \}$.
Then we can replace the right-hand side of \eqref{fNbound} by $O(|z_1| + \dots + |z_m|)$ --- the constant in the big--O notation would depend on $y_1, \ldots, y_m, y$, but not on $z_1, \ldots, z_m$, as long as they are small enough and positive (and $N \geq m+1$).
Plugging \eqref{fNdef} into \eqref{fNbound}, we deduce
\begin{multline}\label{claim2.1}
\B_{a(N)}\left( -\frac{y_1+z_1}{\sqrt{N}}, \ldots, -\frac{y_m+z_m}{\sqrt{N}}, -\frac{y - (z_1 + \ldots + z_m)}{\sqrt{N}} ; N; \theta \right) =\\
\B_{a(N)}\left( -\frac{y_1}{\sqrt{N}}, \ldots, -\frac{y_m}{\sqrt{N}}, -\frac{y}{\sqrt{N}} ; N; \theta \right)
+ O(N^{-\epsilon}), \text{ as }N \rightarrow \infty,
\end{multline}
whenever $(z_1, \ldots, z_m)\in [0, N^{-\epsilon}]^{m}$.

On the other hand, from the definitions $(\ref{defn:FGfns})$ for $G_{\theta, \bfy}(z)$ and $F_{\bfy}(\bfz)$, we have
\begin{equation}\label{claim2.2}
G_{\theta, \bfy}(\bfz) = \prod_{i=1}^m{(y_i - y)^{\theta}} \prod_{1 \leq i < j \leq m}(y_i - y_j)^{2\theta - 1}\cdot (1 + O(N^{-\epsilon})),\quad
F_{\bfy}(\bfz)^{-\theta m - 1} = 1 + O(N^{-\epsilon});
\end{equation}
both bounds are over $(z_1, \ldots, z_m)\in [0, N^{-\epsilon}]^{m}$.

Due to $(\ref{claim2.1})$ and $(\ref{claim2.2})$, Claim 2 is reduced to
\begin{equation}\label{toprove:claim2}
\int_0^{N^{-\epsilon}}\dots \int_0^{N^{-\epsilon}}{F_{\bfy}(\bfz)^{N\theta}\prod_{i=1}^m{(z_i^{\theta - 1}dz_i)}} = \frac{\Gamma(\theta)^m y^{m\theta} {(y_1\cdots y_m)^{\theta}}}{(N\theta)^{m\theta}\prod_{i=1}^m(y_i - y)^{\theta}} \cdot \left( 1 + O(N^{1 - 2\epsilon}) \right).
\end{equation}
Begin with the Taylor expansion around $\bfz = (0^m)$:
\begin{equation*}
F_{\bfy}(\bfz)^{N\theta} = \exp\left( N \theta \ln{ F_{\bfy}(\bfz) } \right)= 
\exp\left( N \theta \left\{ \ln{ F_{\bfy}(0^m)} + \sum_{i=1}^m{z_i\left. \frac{\partial \ln F_{\bfy}(\bfz)}{\partial z_i} \right|_{\bfz =  (0^m)} } + O\left( \sum_{i, j}{|z_iz_j|} \right) \right\} \right).
\end{equation*}
We now calculate $F_{\bfy}(0^m) = 1$ and
\begin{equation*}
\left. \frac{\partial \ln F_{\bfy}(\bfz)}{\partial z_i} \right|_{\bfz = (0^m)} =
\left. \left\{ \frac{1}{z_1 + \ldots + z_m - y} + \frac{1}{z_i + y_i} \right\} \right|_{\bfz = (0^m)} =
- \frac{y_i - y}{yy_i}, \textrm{ for } i = 1, \ldots, m.
\end{equation*}
Therefore, when $0 \leq z_i \leq N^{-\epsilon}$ for $i = 1, 2, \ldots, m$, we have
$$F_{\bfy}(\bfz)^{N\theta} = \prod_{i=1}^m{\exp\left( -\frac{N\theta(y_i - y)}{yy_i}z_i \right)}\times (1 + O(N^{1-2\epsilon})).$$
Back into $(\ref{toprove:claim2})$, we see that its left hand side is the product
$$
\prod_{i=1}^m{\int_0^{N^{-\epsilon}}{\exp\left( -\frac{N\theta (y_i - y)}{yy_i}z_i \right)}z_i^{\theta - 1}dz_i} \times (1 + O(N^{1-2\epsilon})).
$$
In the $i$-th integral above, make the change of variables $z_i = cw_i/N$, for $c = (yy_i)/(\theta (y_i - y))$, so that $w_i$ now ranges from $0$ to $cN^{1-\epsilon}$.
Thus the $i$-th integral above is estimated by
$$\frac{c^{\theta}}{N^{\theta}}\int_0^{cN^{1-\epsilon}}{e^{-w_i}w_i^{\theta - 1}dw_i} \approx
\frac{c^{\theta}}{N^{\theta}} \int_0^{\infty}{e^{-w_i}w_i^{\theta - 1}dw_i} = \frac{c^{\theta}}{N^{\theta}} \Gamma(\theta)
= \frac{\Gamma(\theta)(yy_i)^{\theta}}{(N\theta)^{\theta} (y_i - y)^{\theta}}.$$
The estimate we just made leaves an exponentially small relative error.
Then $(\ref{toprove:claim2})$ is proved.
\end{proof}

\smallskip

\begin{proof}[Proof of Theorem $\ref{thm:asymptoticbessels}$]
Let $\{a(N)\in\PP_N\}_{N \geq 1}$ be a regular sequence of ordered tuples with limiting measure $\mu$.
Because of Lemma $\ref{lem:meanzero}$, we can assume without loss of generality that $\sum_{i=1}^N{a(N)_i} = 0$ for all $N\in\N_+$.
We are going to prove
\begin{equation}\label{eqn:besselslimits}
\lim_{N\rightarrow\infty}{\B_{a(N)}\left( \frac{y_1}{\sqrt{N}}, \dots, \frac{y_m}{\sqrt{N}}; N; \theta \right)} =
\exp\left( \frac{\Var[\mu]}{2\theta}\sum_{i=1}^m{y_i^2} \right),
\end{equation}
uniformly for $(y_1, \ldots, y_m)$ belonging to compact subsets of $\C^m$.

The proof is by induction on $m$.
For $m = 1$, $(\ref{eqn:besselslimits})$ is proved in Subsection $\ref{sec:m1}$.
Now assume that $(\ref{eqn:besselslimits})$ holds for some $m\in\N_+$ (and uniformly on $y_1, \ldots, y_m$), then we prove that it also holds for $(m+1)$.
Let $K = B[0, R]\subset\R^{m+1}$ be the closed ball of radius $R$ inside $\R^{m+1}$.
From Lemma $\ref{lem:boundedness}$, we have
$$
\limsup_{N \rightarrow \infty}\sup_{(y_1, \dots, y_{m+1})\in K}{\left| \B_{a(N)}\left( \frac{y_1}{\sqrt{N}}, \dots, \frac{y_{m+1}}{\sqrt{N}}; N; \theta \right)  \right|} < \infty.
$$
By Montel's theorem, each subsequence of
$$
\left\{  \B_{a(N)}\left( \frac{y_1}{\sqrt{N}}, \dots, \frac{y_{m+1}}{\sqrt{N}}; N; \theta \right)  \right\}_{N \geq m+1}
$$
has a sub-subsequence that converges uniformly on $K$ to some function $\Phi(y_1, \ldots, y_{m+1})$.
The limit $\Phi$ should be an analytic function on the open ball $B(0, R)$.
Moreover, because of Proposition $\ref{prop:masymptotics}$, each such limit should be equal to the right hand side of $(\ref{eqn:besselslimits})$ in the subdomain of $B(0, R)$ that is defined by the inequalities
\begin{equation*}
\left\{
\begin{aligned}
&y_1 > \dots > y_m > y_{m+1} > 0,\\
&y_1 - y_2 > y_3,  \dots, y_{m-2} - y_{m-1} > y_m, \ y_{m-1} - y_m > y_{m+1}.
\end{aligned}
\right.
\end{equation*}
Since the right hand side of $(\ref{eqn:besselslimits})$ is an entire function, analytic continuation shows that each function $\Phi(y_1, \ldots, y_{m+1})$ must be equal to the right hand side of $(\ref{eqn:besselslimits})$ in the open ball $B(0, R)$.
Since the radius $R>0$ was arbitrary, the result follows.
\end{proof}


\begin{thebibliography}{9}

\bibitem{A}
G. W. Anderson. A short proof of Selberg’s generalized beta formula. In Forum Math., vol. 3, no. 3, pp. 415-418. Walter de Gruyter, Berlin/New York, 1991.

\bibitem{AGZ}
G. W. Anderson, A. Guionnet, O. Zeitouni. An Introduction to Random Matrices. Cambridge University Press, 2010.

\bibitem{An}
J-P Anker. An introduction to Dunkl theory and its analytic aspects. Analytic, Algebraic and Geometric Aspects of Differential Equations. Birkhäuser, Cham, 2017, pp. 3-58.

\bibitem{Ba}
Y. Baryshnikov. GUEs and queues. Probab. Theory Related Fields 119, no. 2 (2001), pp. 256-274.

\bibitem{BE}
H. Bateman, and A. Erdelyi. Higher transcendental functions. Vol. 1. No. 2. New York: McGraw-Hill, 1953.

\bibitem{BFG}
F. Bekerman, A. Figalli, and A. Guionnet. Transport Maps for $\beta$-Matrix Models and Universality. Comm. Math. Phys. 338, no. 2 (2015), pp. 589-619.

\bibitem{BCG}
F. Benaych-Georges, C. Cuenca and V. Gorin. In preparation.

\bibitem{B}
R. Bhatia. Matrix analysis, volume 169 of Graduate Texts in Mathematics, 1997.

\bibitem{BC}
A. Borodin, and I. Corwin. Macdonald processes. Probab. Theory Related Fields 158, no. 1-2 (2014), pp. 225-400.

\bibitem{BG}
A. Borodin, and V. Gorin. General $\beta$-Jacobi Corners Process and the Gaussian Free Field. Comm. Pure Appl. Math. 68, no. 10 (2015), pp. 1774-1844.

\bibitem{BGu}
G. Borot, and A. Guionnet. Asymptotic expansion of $\beta$ matrix models in the one-cut regime. Comm. Math. Phys. 317, no. 2 (2013), pp. 447-483.

\bibitem{BEY}
P. Bourgade, L. Erd\"{o}s, and H. T. Yau. Universality of general $\beta$-ensembles. Duke Math. J. 163, no.6 (2014), pp. 1127-1190.

\bibitem{CE}
G. Cipolloni, and L. Erd\"{o}s. Fluctuations for linear eigenvalue statistics of sample covariance random matrices. Preprint, arXiv:1806.08751 (2018).

\bibitem{C}
E. T. Copson. Asymptotic expansions. No. 55. Cambridge University Press, 2004.

\bibitem{Cu18b}
C. Cuenca. Pieri integral formula and asymptotics of Jack unitary characters. Selecta Math. (N.S.) (2017), pp. 1-53.

\bibitem{Cu18a}
C. Cuenca. Asymptotic Formulas for Macdonald Polynomials and the Boundary of the $(q, t)$--Gelfand-Tsetlin Graph. SIGMA Symmetry Integrability Geom. Methods Appl. 12 (2018), 001.

\bibitem{dJ}
M. FE. de Jeu. The dunkl transform. Invent. Math. 113, no. 1 (1993), pp. 147-162.

\bibitem{De}
M. Defosseux. Orbit measures, random matrix theory and interlaced determinantal processes. Ann. Inst. Henri Poincaré Probab. Stat. Vol. 46. No. 1. Institut Henri Poincaré, 2010.

\bibitem{Di}
E. Dimitrov. Six-vertex models and the GUE-corners process. Int. Math. Res. Not. IMRN (2016).

\bibitem{DE}
I. Dumitriu, and Alan Edelman. Global spectrum fluctuations for the $\beta$-Hermite and $\beta$-Laguerre ensembles via matrix models. J. Math. Phys. 47, no. 6 (2006): 063302.

\bibitem{D}
C. F. Dunkl. Differential-difference operators associated to reflection groups. Trans. Amer. Math. Soc. 311, no. 1 (1989), pp. 167-183.

\bibitem{ES}
A. Edelman, and B. D. Sutton. From random matrices to stochastic operators. J. Stat. Phys. 127, no. 6 (2007), pp. 1121-1165.

\bibitem{ES}
L. Erd\"{o}s, and D. Schröder. Fluctuations of rectangular Young diagrams of interlacing Wigner eigenvalues. Int. Math. Res. Not. IMRN 2018, no. 10 (2017), pp. 3255-3298.

\bibitem{FR}
P. J. Forrester, and E. M. Rains. Interpretations of some parameter dependent generalizations of classical matrix ensembles. Probab. Theory Related Fields 131, no. 1 (2005), pp. 1-61.

\bibitem{FW}
P. Forrester , and S. V. E. N. Warnaar. The importance of the Selberg integral. Bull. Amer. Math. Soc. (N.S.) 45, no. 4 (2008), pp. 489-534.

\bibitem{F}
P. J. Forrester. Log-gases and random matrices (LMS-34). Princeton University Press, 2010.

\bibitem{FIZ}
P. J. Forrester, J. R. Ipsen, D. Z, and L. Zhang. Orthogonal and symplectic Harish-Chandra integrals and matrix product ensembles. Random Matrices Theory Appl. (2019): 1950015.

\bibitem{Ga}
F. R. Gantmakher. Lectures in analytical mechanics. Moscow, Nauka, 1966.

\bibitem{GN}
I. M. Gelfand, and M. A. Naimark. Unitary representations of the classical groups. Tr. Mat. Inst. Steklova 36 (1950), pp. 3-288.

\bibitem{G}
V. Gorin. From alternating sign matrices to the gaussian unitary ensemble. Comm. Math. Phys. 332, no. 1 (2014), pp. 437-447.

\bibitem{GoM}
V. Gorin, and A. W. Marcus. Crystallization of random matrix orbits. Int. Math. Res. Not. IMRN (2018).

\bibitem{GP}
V. Gorin, and G. Panova. Asymptotics of symmetric polynomials with applications to statistical mechanics and representation theory. Ann. Probab. 43, no. 6 (2015), pp. 3052-3132.

\bibitem{GS}
V. Gorin, and M. Shkolnikov. Multilevel Dyson Brownian motions via Jack polynomials. Probab. Theory Related Fields  163, no. 3-4 (2015), pp. 413-463.

\bibitem{GZ}
V. Gorin, and L. Zhang. Interlacing adjacent levels of $\beta$–Jacobi corners processes. Probab. Theory Related Fields  (2016), pp. 1-67.

\bibitem{GK}
T. Guhr, and H. Kohler. Recursive construction for a class of radial functions. I. Ordinary space. J. Math. Phys. 43, no. 5 (2002), pp. 2707-2740.

\bibitem{GH}
A. Guionnet, and J. Huang. Rigidity and Edge Universality of Discrete $\beta$-Ensembles. Comm. Pure Appl. Math. (2017).

\bibitem{GuM}
A. Guionnet, and M. Maıda. A Fourier view on the R-transform and related asymptotics of spherical integrals. J. Funct. Anal. 222, no. 2 (2005), pp. 435-490.

\bibitem{HC}
Harish-Chandra. Differential operators on a semisimple Lie algebra. Amer. J. Math. (1957), pp. 87-120.

\bibitem{IZ}
C. Itzykson, and J-B. Zuber. The planar approximation. II. J. Math. Phys. 21, no. 3 (1980), pp. 411-421.

\bibitem{J2}
K. Johansson. On fluctuations of eigenvalues of random Hermitian matrices. Duke Math. J. 91, no. 1 (1998), pp. 151-204.

\bibitem{J}
K. Johansson. Non-intersecting paths, random tilings and random matrices. Probab. Theory Related Fields 123, no. 2 (2002), pp. 225-280.

\bibitem{JN}
K. Johansson, and E. Nordenstam. Eigenvalues of GUE minors. Electron. J. Probab. 11 (2006), pp. 1342-1371.

\bibitem{K}
K. WJ Kadell. The Selberg–Jack symmetric functions. Adv. Math. 130, no. 1 (1997), pp. 33-102.

\bibitem{M82}
I. G. Macdonald. Some conjectures for root systems. SIAM J. Math. Anal. 13, no. 6 (1982), pp. 988-1007.

\bibitem{M}
I. Macdonald. Symmetric Functions and Hall Polynomials. Oxford University Press, Oxford, second edition, 1999.

\bibitem{M13}
I. G. Macdonald. Hypergeometric functions I. Preprint, arXiv:1309.4568 (2013).

\bibitem{MM}
E. S. Meckes, and M. W. Meckes. Random matrices with prescribed eigenvalues and expectation values for random quantum states. Preprint, arXiv:1711.02710 (2017).

\bibitem{Me}
M. L. Mehta. Random Matrices (3rd ed.). Amsterdam: Elsevier/Academic Press, 2004.

\bibitem{MP}
S. Mkrtchyan, and L. Petrov. GUE corners limit of $q$--distributed lozenge tilings. Electron. J. Probab. 22 (2017).

\bibitem{N}
A. Yu. Neretin. Rayleigh triangles and non-matrix interpolation of matrix beta integrals. Sb. Math. 194, no. 4 (2003): 515.

\bibitem{Ni}
P. P. Nikitin. $O(\infty)$-- and $Sp(\infty)$--invariant ergodic measures on the spaces of infinite antisymmetric and quaternionic antihermitian matrices. Zap. Nauchn. Sem. S.-Peterburg. Otdel. Mat. Inst. Steklov. (POMI) 437 (2015), pp. 207-220.

\bibitem{No}
J. Novak. Lozenge tilings and Hurwitz numbers. J. Stat. Phys. 161, no. 2 (2015), pp. 509-517.

\bibitem{OkOl}
A. Okounkov, and G. Olshanski. Shifted Jack Polynomials, Binomial Formula, and Applications. Math. Res. Lett. 4 (1997), pp. 69-78.

\bibitem{OkR}
A. Okounkov, and N. Reshetikhin. Correlation function of Schur process with application to local geometry of a random 3-dimensional Young diagram. J. Amer. Math. Soc. 16, no. 3 (2003), pp. 581-603.

\bibitem{OR2}
A. Y. Okounkov, and N. Y. Reshetikhin. The birth of a random matrix. Mosc. Math. J. 6, no. 3 (2006), pp. 553-566.

\bibitem{OV}
G. Olshanski, and A. Vershik. Ergodic unitarily invariant measures on the space of infinite Hermitian matrices. Amer. Math. Soc. Transl. Ser 2 175 (1996), pp. 137-176.

\bibitem{O}
F. Olver. Asymptotics and special functions. AK Peters/CRC Press, 1997.

\bibitem{Op}
E. M. Opdam. Dunkl operators, Bessel functions and the discriminant of a finite Coxeter group. Compos. Math. 85, no. 3 (1993), pp. 333-373.

\bibitem{RRV}
J. Ramirez, B. Rider, and B. Virág. Beta ensembles, stochastic Airy spectrum, and a diffusion. J. Amer. Math. Soc. 24, no. 4 (2011), pp. 919-944.

\bibitem{Sh}
M. Shcherbina. Change of variables as a method to study general $\beta$-models: bulk universality. J. Math. Phys. 55, no. 4 (2014): 043504.

\bibitem{St}
R. P. Stanley. Some combinatorial properties of Jack symmetric functions. Adv. Math. 77, no. 1 (1989), pp. 76-115.

\bibitem{S}
Y. Sun. Matrix models for multilevel Heckman-Opdam and multivariate Bessel measures. Preprint, arXiv:1609.09096 (2016).

\bibitem{TE}
F. Tricomi, and A. Erdélyi. The asymptotic expansion of a ratio of gamma functions. Pacific J. Math. 1, no. 1 (1951), pp. 133-142.

\bibitem{VV}
B. Valkó, and B. Virág. Continuum limits of random matrices and the Brownian carousel. Invent. Math. 177, no. 3 (2009), pp. 463-508.

\end{thebibliography}
\end{document}